\documentclass[a4paper,12pt]{amsart}
\usepackage{amsthm, amsmath,amsfonts,amssymb,amstext,amscd,mathabx,mathdots,mathtools}
\usepackage{graphicx,epsf}
\usepackage{hyperref}
\usepackage{xcolor}
\usepackage{calrsfs}
\usepackage{tikz}
\usepackage{pgfplots}
\pgfplotsset{compat=1.18}

\newtheorem{theorem}{Theorem}[section]

\newtheorem{lemma}[theorem]{Lemma}

\newtheorem{remark}[theorem]{Remark}

\numberwithin{equation}{section}

\title[Pullback dynamics for a semilinear heat equation]{Pullback dynamics for a semilinear heat equation with homogeneous Neumann  boundary conditions on time-varying domains}

\author[G. S. Arag\~ao]{Gleiciane S. Arag\~ao$^{1,*}$}\thanks{$^*$ Corresponding author}\thanks{$^1$Research partially supported by FAPESP Nº 2020/14075-6, Brazil}
\address[G. S. Arag\~ao]{
Departamento de Ensino em Ci\^encias e Matem\'atica, Universidade Federal de S\~ao Paulo, Rua S\~ao Nicolau, 210, Diadema-SP, Cep 09913-030, Brazil.}
\email{gleiciane.aragao@unifesp.br}

\author[F. D. M. Bezerra]{Flank D. M. Bezerra$^{2}$}\thanks{$^2$Research partially supported by
CNPq Nº 303039/2021-3, Brazil}
\address[F. D. M. Bezerra]{Departamento de Matem\'atica, Universidade Federal da Para\'{\i}ba, Cidade Universit\'{a}ria-Campus I, Jo\~{a}o Pessoa-PB, Cep 58051-900, Brazil.}
\email{flank@mat.ufpb.br}

\author[L. G. Mendonça]{Lucas G. Mendonça$^{3}$}\thanks{$^3$Research partially supported by
CNPq Nº 141176/2018-0, Brazil}
\address[L. G. Mendonça]{Secretaria Municipal de Educação da Prefeitura de São Paulo, Rua Dr. Diogo Faria, 1247, São Paulo-SP, Cep 04037-004, Brazil.}
\email{luk.galhego@gmail.com}

\date{\today}

\begin{document}

\begin{abstract}
We are interested in studying a non-autonomous semilinear heat equation with homogeneous Neumann boundary conditions on time-varying domains. 
Using a differential geometry approach with coordinate transformations technique, we will show that the non-autonomous problem on a time-varying domain is equivalent, in some sense, to a non-autonomous problem on a fixed domain. Furthermore, we intend to show the local existence and uniqueness of solutions to this problem, as well as, to extend these solutions globally.
Finally, we will show the existence of pullback attractors. To the best of our knowledge,  results on attractors are new even for non-autonomous semilinear heat equations with homogeneous Neumann boundary conditions on time-varying domains subject to conditions with more restrictive assumptions.

\vspace{0.3cm} 


\vspace{0.3cm}


\end{abstract}

\maketitle

\tableofcontents

%%%%%%%%%%%%%%%%%%%%%%%%%%%%%%%%%%%%%

\section{Introduction}

\vspace{0.2cm}

\subsection{Seeting of previous works}

In this work, we are interested in the study  of the pullback attractors for a non-autonomous semilinear parabolic partial differential equation with homogeneous Neumann boundary conditions on time-varying domains. A pioneering bibliography for non-autonomous parabolic equations can be found in \cite{Sobolevskii}. Other more recent works of this theory can be found in \cite{ NolascoNaoAutonomo2,NolascoAtratores,NolascoExistencia, Daners,Pazy}. Semilinear parabolic problems with time-varying domains are intrinsically non-autonomous, even if the terms in the problem do not explicitly depend on time.

Problems with domain varying with a parameter, which is not time, have been considered by several authors. For example, in
\cite{Barbosa, Pereira}, the authors studied the continuity of attractors, however, initially, they applied the technique of \cite{Henry2} to rewrite the original problem, which is posed on a perturbed domain in some fixed domain. In \cite{SimoneAragao,SimoneArrieta}, the authors used another technique, the concept of $E$-convergence, which has as its key ingredient the use of extension operators, to be able to compare functions or operators defined in different spaces and continue with the domains varying. In this work, we will choose, in some sense, by the technique of \cite{Henry2}.

The work \cite{Fridman} is one of the first on the subject of problems with domain varying with the time parameter. In this article, 
the author studied
the asymptotic behavior of solutions of parabolic equations of any order, using strong convergence hypotheses (uniform convergence) and a very high regularity for the coefficients and nonlinearities of the equation. Our work is different from this, since we study the problem with less regular hypotheses and in another phase spaces. Furthermore, in the nonlinear case, 
the author did not show the existence of solutions, he only studied
the asymptotic behavior considering that the problem has a solution. 
 
Works more recent on parabolic problems with non-cylindrical and time-varying domains are given in \cite{PeterKloeden2,PeterKloeden}, respectively, where the authors studied a semilinear heat equation with homogeneous Dirichlet boundary conditions. The authors showed the existence of pullback attractors for the non-autonomous problem generated by the variational solutions, using Sobolev-Bockner spaces and techniques distinct from our problem. Moreover,  \cite{PeterKloeden} used 
a similar technique to  \cite{Henry2}, to rewrite the original problem, which is posed on a time-varying domain, into some problem on a fixed domain.

We can also mention other works in this line, such as \cite{Matofu}, where the existence of pullback attractors was studied for a wave equation, with
homogeneous Dirichlet boundary conditions, defined on domains of $\mathbb{R}^3$ that are varying with the time. Recently, in \cite{Lopes}, the authors studied the existence, uniqueness and asymptotic behavior of a Laplace equation on non-cylindrical domains.

Based on the works cited, this work proposes the study of a problem that involves the topics above: the pullback attractors of a semilinear heat equation 
with homogeneous Neumann boundary conditions on time-varying domains. Therefore, this work extends the results obtained in \cite{PeterKloeden} to the case of Neumann boundary conditions. It is important to highlight 
that Neumann boundary conditions with the abstract formulation of the problem were not considered in the works cited above. Moreover, we will also work with the nonlinear part of the equation depending on the gradient, that is, $f=f(t,u,\nabla u),$ according to \cite{Daners}.

\subsection{Main results}

To describe the problem, let $\mathcal{O}$ be a nonempty bounded open subset of $\mathbb{R}^{n}$, $n\geq 3$, 
with $C^{2}$ boundary $\partial \mathcal{O}$. We consider the function
    \begin{equation}
    \nonumber
    \begin{array}{rccl}
    \displaystyle r :&   \mathbb{R} \times \overline{\mathcal{O}} &\rightarrow& \mathbb{R}^{n} \\ 
    \displaystyle 	 &   (t,y) & \mapsto & r(t,y)
    \end{array}
    \end{equation}
such that $r\in C^{1}( \mathbb{R}\times \overline{\mathcal{O}},\mathbb{R}^{n})$ and     
\[
r(t,\cdot):\overline{\mathcal{O}} \rightarrow \overline{\mathcal{O}}_{t}
\]
is a diffeomorphism of  class $C^{2}$ for all $t \in \mathbb{R}$, where  $\mathcal{O}_{t}=r(t,\mathcal{O})$. We denote by
 $r^{-1}(t,\cdot)$ the inverse function of
 $r(t,\cdot)$, with $r(t,y)=(r_1(t,y),r_2(t,y),...,r_n(t,y))$, for all $(t,y)\in \mathbb{R}\times\overline{\mathcal{O}}$,  and $r^{-1}(t,x)=(r^{-1}_1(t,x),r^{-1}_2(t,x),...,r^{-1}_n(t,x))$, for all $(t,x)\in \mathbb{R}\times\overline{\mathcal{O}}_t$, see Figure \ref{FigguuA}.

\begin{figure}[!htp]
    \centering
\begin{tikzpicture}
\draw[rotate around={100:(0,0)},shift={(9,-0.5)},smooth] plot[smooth] coordinates{
    (-1.8,-0.2) (-1.6,0.5) (-1,0.9)
    (-0.3,1) (1,0.7) (2,0.2) (2.22,0) (2.42,-0.45) (2.25,-1)
};
\draw[dotted,rotate around={100:(0,0)},shift={(9,-0.5)},smooth] plot[smooth] coordinates{
  (-1.8,-0.2) (-1.6,0.5) (-1,0.9)
    (-.3,1) (1,0.7) (2,0.2) (2.22,0)
    (2.42,-0.45) (2.25,-1)
    (1.8,-1.18) (1.2,-1)
    (0,-0.8)
    (-.9,-0.9) (-1.5,-0.7)
};
\draw[rotate around={100:(0,0)},shift={(7.75,-7.5)},smooth] plot[smooth cycle] coordinates{
    (-1.8,-0.2) (-1.6,0.5) (-1,0.9)
    (-.3,1) (1,0.7) (2,0.2) (2.22,0)
    (2.42,-0.45) (2.25,-1)
    (1.8,-1.18) (1.2,-1)
    (0,-0.8)
    (-.9,-0.9) (-1.5,-0.7)
};
\draw[rotate around={100:(0,0)},shift={(8.35,-4)},smooth] plot[smooth] coordinates{
    (-1.8,-.2) (-1.6,0.5) (-1,0.9)
    (-.3,1) (1,0.7) (2,0.2) (2.22,0)
    (2.42,-0.45) (2.25,-1)
};
\draw[dotted,rotate around={100:(0,0)},shift={(8.35,-4)},smooth] plot[smooth] coordinates{
    (-1.8,-.2) (-1.6,0.5) (-1,0.9)
    (-.3,1) (1,0.7) (2,0.2) (2.22,0)
    (2.42,-0.45) (2.25,-1) (1.8,-1.18) (1.2,-1)
    (0,-0.8)
    (-.9,-0.9) (-1.5,-0.7)
};
\node at (-2.3,8){$\mathcal{O}$};
\node at (1.3,8){$\mathcal{O}$};
\node at (4.8,8){$\mathcal{O}$};
\draw(-0.8,7.2)--(6.2,7.2);
\draw(-0.8,11.43)--(6.2,11.43);
\draw[->](-2.5,9)--(8,9) node[right]{$\mathbb{R}$};
\draw(-1,6)--(-1,13) node[right]{$\mathbb{R}^n$};
\draw[<-] plot [smooth] coordinates {(4.5,6) (4.7,5) (4.5,4)};
\node at (5.3,5){$r^{-1}$};
\node at (-0.2,5){$r$};
\draw[->] plot [smooth] coordinates {(0.4,6) (0.2,5) (0.4,4)};
\draw[->](-2.5,0)--(8,0) node[right]{$\mathbb{R}$};
\draw(-1,-3)--(-1,4) node[right]{$\mathbb{R}^n$};
\draw(-1,-2.5)arc(270:90:0.5 and 2.5);
\draw[dotted](-1,2.5)arc(-270:90:0.5 and 2.5);
\draw[smooth,variable=\x,domain=-1:pi] plot(\x,{2.7+0.3*sin(2.5*\x r)});
\draw[smooth,variable=\x,domain=-1:pi] plot(\x,{-2.35+0.3*sin(2.5*\x r)});
\draw[smooth,variable=\x,domain=0.75*pi:1.8*pi] plot(\x+0.8,{2.35+0.65*cos(2.7*\x r)});
\draw[smooth,variable=\x,domain=0.75*pi:1.8*pi] plot(\x+0.8,{-1.42-0.65*cos(2.7*\x r)});
\draw(3.21,-2.05)arc(270:89:0.5 and 2.5);
\draw[dotted](3.21,3)arc(-270:90:0.5 and 2.5);
\draw(6.5,-0.85)arc(270:89:0.3 and 1.3);
\draw(6.5,1.8)arc(-270:90:0.3 and 1.3);
\node at (-1.8,-1){$\mathcal{O}_0$};
\node at (2.4,-0.8){$\mathcal{O}_\tau$};
\node at (5.9,-0.6){$\mathcal{O}_t$};
\node at (3.2,0){$\bullet$};
\node at (3.2,-0.4) {$\tau$};
\node at (6.5,0){$\bullet$};
\node at (6.5,-0.4) {$t$};
\node at (2.5,9){$\bullet$};
\node at (2.5,8.6) {$\tau$};
\node at (6,9){$\bullet$};
\node at (6,8.6) {$t$};
\end{tikzpicture}
    \caption{Domains}
    \label{FigguuA}
\end{figure}

Let $t_{0} \in \mathbb{R}$, we denote
$$ \mathcal{O}_{t_{0}}=\left\{r(t_{0},y):y \in \mathcal{O}\right\} \quad \mbox{and} \quad \partial \mathcal{O}_{t_{0}}=\left\{r(t_{0},y):y \in \partial \mathcal{O}\right\}.
$$

We define 
$$
Q_{\tau,T}=\bigcup_{t\in(\tau,T)}{ \left\{ t\right\}\times \mathcal{O}_{t}} \quad \mbox{and
} \quad \Sigma_{\tau,T}=\bigcup_{t\in(\tau,T)}{\left\{ t\right\}\times \partial \mathcal{O}_{t}}, \quad \mbox{for all $T>\tau$},
$$
and we denote by $Q_{\tau}=Q_{\tau,\infty}$ and  $\Sigma_{\tau}=\Sigma_{\tau,\infty}$.

For any $T>\tau$, we observe that $Q_{\tau,T}$ is an open subset in $\mathbb{R}^{n+1}$ with boundary $$\partial Q_{\tau,T}=\Sigma_{\tau,T} \cup \left(\mathcal{O}_{\tau}\times \left\{\tau\right\}\right) \cup \left(\mathcal{O}_{T}\times \left\{T\right\}\right).
$$

In this paper, we are particularly interested in studying the following non-autonomous semilinear heat equation with homogeneous Neumann boundary conditions
\begin{equation}
\label{equacaoproblema}
\left\{
\begin{array}{lr}
\displaystyle \frac{\partial u}{\partial t}(t,x) -\Delta u(t,x)+\beta u(x)=f(t,u), & (t,x) \in Q_{\tau} \\ 
\displaystyle \frac{\partial u}{\partial n_{t}}(t,x)=0, & (t,x) \in \Sigma_{\tau} \\
\displaystyle u(\tau,x)=u_{\tau}(x), & x \in \mathcal{O}_{\tau}
\end{array}
\right.
\end{equation}
where $\beta>0$, $n_{t}(x)$ is the unit outward  normal vector at $x \in \partial \mathcal{O}_{t}$, $\tau \in \mathbb{R}$, $u_{\tau}:\mathcal{O}_{\tau} \rightarrow \mathbb{R}$ and $f: \mathbb{R}^2 \rightarrow \mathbb{R}$ is a  nonlinear function. We will show the local existence and uniqueness of solutions to the problem \eqref{equacaoproblema}, as well as extend these solutions globally, and then we will show the existence of pullback attractors.

The way we put our original problem \eqref{equacaoproblema}, we are treating a problem that for each $t>\tau$ the equation resides in a space $Q_{\tau}\subset \mathbb{R}^{n+1}$, which is varying with the parameter $t$. We will apply the variable change technique to
 rewrite the original problem 
as an auxiliary non-autonomous problem on the fixed domain $\mathcal{O}$.

Next, we will state the hypotheses that will be necessary in this paper. The first hypothesis will be fundamental to guarantee that the abstract operator associated to the elliptic operator
in \eqref{equacaoproblema} is uniformly sectorial and uniformly H\"older continuous, whose domains are independent of time in bounded intervals.

\vspace{0.4cm}

\noindent \textbf{(H1) Conditions about diffeomorphism:} Suppose $r^{-1}(\cdot,x)\in C^{1}(\mathbb{R},\mathbb{R}^n)$ for all $x\in \overline{\mathcal{O}}_{t}$ and that there are functions $h \in C(\mathbb{R})$,  and $p_{ik} \in C^{1}(\mathbb{R}^{n},\mathbb{R})$ such that
$$\frac{\partial r^{-1}_{k}}{\partial x_i}(t,r(t,y))=h(t)p_{ik}(y), \quad \mbox{for all $t \in I$ and  $y \in \overline{\mathcal{O}}$},
$$
where $i,k=1,...,n$ and $r^{-1}_{k}$ is the projection
 of diffeomorphism  $r^{-1}$ in $k$-th coordinate. Also, suppose that the function $h$ is Hölder continuous with exponent
 $\theta \in (0,1]$ and that there are positive constants $h_0$ and $h_1$ such that $0<h_0\leq h(t)\leq h_1$, for all $t\in \mathbb{R}$.

\vspace{0.4cm}

To obtain the local existence and uniqueness of solutions for the problem (\ref{equacaoproblema}), we will need to study nonlinearity and its respective abstract application. For this, we assume in addition the following condition:

\vspace{0.4cm}

\noindent \textbf{(H2) 
Growth and regularity conditions of nonlinearity:}
Let $n\geq 3$, $\displaystyle \alpha \geq \frac{1}{2}$ and  $I\subset \mathbb{R}$ be a  bounded interval and suppose $f(\cdot,u) \in C(\mathbb{R})$, for all $u\in \mathbb{R}$, $f(t,\cdot) \in C^{1}(\mathbb{R})$ for all $t\in I$, and $f:\mathbb{R}^2 \to \mathbb{R}$ satisfies the following growth condition:
$$
|\partial_u f(t,u)|\leq c(1+|u|^{\rho}), \quad \mbox{for all $t\in I$ and $u\in \mathbb{R}$,}
$$
for some constant $c>0$ independent of $t$ and with growth coefficient
$\displaystyle 0 < \rho \leq \frac{4\alpha}{n-4\alpha}.$

\vspace{0.4cm}

In particular, using the Mean Value Theorem, the growth condition in
\textbf{(H2)} implies the following condition:
\begin{equation}
\label{cres02}
|f(t,u)-f(t,v)|\leq c |u-v|(1+|u|^{\rho} + |v|^{\rho}), \quad  \mbox{for all $t\in I$ and $u,v\in \mathbb{R}$.}
\end{equation}

Note that if $\displaystyle \alpha=\frac{1}{2}$, that is, to study the problem
\eqref{equacaoproblema} with initial data in
 $u_{\tau}\in H^{1}(\mathcal{O})$, the growth coefficient in
\textbf{(H2)} or \eqref{cres02} satisfies
 $\displaystyle 0< \rho \leq \frac{2}{n-2}.$ In particular, for $n= 3$ we have $0<\rho \leq 2$.

\vspace{0.4cm}

To ensure the boundedness of solutions in bounded time intervals and the global existence, we will need a sign condition:

\vspace{0.4cm}

\noindent \textbf{(H3) Sign condition of nonlinearity:}
Supoose $f:\mathbb{R}^2 \to \mathbb{R}$ satisfies the following sign condition:
$$
|f(t,u)|\leq k_1|u|+k_2, \quad \mbox{for all $t\in I$ and $u\in \mathbb{R}$,}
$$
for some constants $k_1,k_2>0$ independent of $t$, where $I\subset \mathbb{R}$ is any bounded interval.

\vspace{0.4cm}

To study pullback attractors, in addition to the previous hypotheses, we will assume an additional hypothesis, which in a way can be interpreted as a hypothesis of convergence about domains that vary with time: 

\vspace{0.4cm}

\noindent \textbf{(H4) Asymptotic behavior of diffeomorphism:} Suppose that
$$|h(t)-h(\tau)| \to 0, \quad \mbox{as $t-\tau \to \infty$}.
$$

\vspace{0.2cm}

Moreover, \textbf{(H4)} will guarantee
 the exponential decay of the evolution process of the homogeneous abstract parabolic problem associated to  semilinear parabolic problem (\ref{equacaoproblema}).

 We emphasize that the authors \cite{PeterKloeden,Matofu} work with linear and separable diffeomorphism. All the work \cite{Matofu} is done for this type of diffeomorphism, in the case of a wave equation with domain in $\mathbb{R}^3$. Although our work follows the same idea, we require in \textbf{(H1)} that each derivative of the projection of the inverse of the diffeomorphism composed with the diffeomorphism is only separable, that is, we have a weaker condition.

With the hypotheses described above and based on the technique of \cite{Henry2,PeterKloeden}, we will rewrite the original problem (\ref{equacaoproblema}), which is posed on the domain $\mathcal{O}_t$ (varying with the parameter $t$), as an auxiliary problem on the fixed domain $\mathcal{O} $, using coordinate transformation. This way, we will analyze this problem with a fixed domain. We will prove the existence and uniqueness of solutions and the pullback dynamics.

On well-posedness and asymptotic dynamics results for  Neumann problems for the PDE's on non-cylindrical domains, we can refer to \cite{BELLUZI2022808, Hofmann1998,25370ef8-2358-396e-8486-e46ab4a56634}. But, to the best of our knowledge,  results on attractors are new even for non-autonomous semilinear heat equations with homogeneous Neumann boundary conditions on time-varying domains subject to conditions with more restrictive assumptions.

\subsection{Outline of the paper}

The paper is structured as follows.
In Section \ref{sec:mudancadecoordenada}, we will do a coordinate transformation to rewrite the original problem (\ref{equacaoproblema}), which is posed on a time-varying, into some problem on a fixed domain. In Section \ref{sec:problemaabstrato}, we will write this problem with the fixed domain as an abstract parabolic evolution problem in some Banach space. Then, we will verify that this abstract problem is well posed, that is, we will prove the local existence and uniqueness of solutions. Consequently, we will be obtaining the existence and uniqueness of the original problem
(\ref{equacaoproblema}). For this, we will use the hypotheses \textbf{(H1)} and \textbf{(H2)}. In Section \ref{sec:existenciaglobal}, assuming additionally the hypothesis \textbf{(H3)}, we will prove that the solutions are globally defined. In Section \ref{Sec_principal_Atrator_Pullback}, we will show that there is a family of pullback attractors. To do so, in addition to the previous hypotheses, it will be necessary to assume the hypothesis \textbf{(H4)}. Finally, in Section \ref{Sec_Final_Ex_Difeo}, we will explore an application in $\mathbb{R}^3$ 
with linear and separable diffeomorphism.

\vspace{0.2cm}

%%%%%%%%%%%%%%%%%%%%%%%%%%%%%%%%%%%%%

\section{Coordinate transformations}
\label{sec:mudancadecoordenada}

\vspace{0.2cm}

This technique uses a parameterization given by the diffeomorphism, $r(t,\cdot)$, defined above and can only be done because of his regularity for all $t\in\mathbb{R}.$

Let $[\tau,T]$ be a finite time interval, that is,  we will initially study the local problem of (\ref{equacaoproblema}) at some $Q_{\tau,T}$, defined as above. Consider the following parametrization
$$
v(t,y)=u(t,r(t,y)), \quad \mbox{for any $(t,y) \in [\tau,T] \times \mathcal{O}$,}
$$
or equivalently,
$$
u(t,x)=v(t,r^{-1}(t,x)), \quad \mbox{for any $(t,x) \in [\tau,T] \times \mathcal{O}_{t}$.} 
$$

Let us suppose that $u(t,x)$ is a classical solution of (\ref{equacaoproblema}) in the sense of classical derivative, we fix $\displaystyle  i=1,...,n$, to find the relation between the derivates of $\displaystyle v$ and $u$ and rewrite the problem (\ref{equacaoproblema}) in a fixed domain.
 
The first spatial derivative is given by
$$\displaystyle \frac{\partial u( t,x)}{\partial x_{i}} =\sum _{j=1}^{n}\frac{\partial v( t,y)}{\partial y_{j}}\frac{\partial r_{j}^{-1}( t,x)}{\partial x_i}.$$

The second spatial derivative is given by
$$\displaystyle \frac{\partial ^{2} u( t,x)}{\partial x_{i}^{2}} =\sum _{j=1}^{n}\frac{\partial v( t,y)}{\partial y_{j}}\frac{\partial ^{2} r_{j}^{-1}( t,x)}{\partial x_i^{2}} +\sum _{k,j=1}^{n}\frac{\partial ^{2} v( t,y)}{\partial y_{j} \partial y_{k}}\frac{\partial r_{j}^{-1}( t,x)}{\partial x_i}\frac{\partial r_{k}^{-1}( t,x)}{\partial x_i}.$$

The time derivative is given by
$$\displaystyle \frac{\partial u( t,x)}{\partial t} =\frac{\partial v( t,y)}{\partial t} +\sum _{j=1}^{n}\frac{\partial r_{j}^{-1}( t,r( t,y))}{\partial t}\frac{\partial v( t,y)}{\partial y_{j}}.$$
Then, 
\begin{eqnarray}
\nonumber \displaystyle \Delta u(t,x) = \sum_{i=1}^{n}{\frac{\partial^{2} u(t,x)}{\partial x^{2}_{i}}} =\sum _{j=1}^{n}\frac{\partial v( t,y)}{\partial y_{j}} \Delta _{x} r_{j}^{-1}( t,x) + \sum _{k,j=1}^{n}\frac{\partial ^{2} v( t,y)}{\partial y_{j} \partial y_{k}}\sum _{i=1}^{n}\frac{\partial r_{j}^{-1}( t,x)}{\partial x_i}\frac{\partial r_{k}^{-1}( t,x)}{\partial x_i},
\end{eqnarray}
where $\Delta_{x}$ denotes the Laplacian operator in variable $\displaystyle x=(x_{1},...,x_{n})$  and $r^{-1}_{k}$ is the $k$-projection of $r^{-1}$ in $k$-th coordinate, $k=1,...,n$.

Since $x=r(t,y)$, let us  denote by
\begin{equation}
\label{ajk}
a_{jk}( t,y) =\sum _{i=1}^{n}\frac{\partial r_{j}^{-1}( t,r( t,y))}{\partial x_i}\frac{\partial r_{k}^{-1}( t,r( t,y))}{\partial x_i}, \quad  \textrm{for  $j,k=1,...,n$.}
\end{equation}
Thus,
$$
\displaystyle \Delta u( t,x) =\sum _{j=1}^{n}\frac{\partial v( t,y)}{\partial y_{j}} \Delta _{x} r_{j}^{-1}( t,r( t,y)) +\sum _{k,j=1}^{n} a_{jk}( t,y)\frac{\partial ^{2} v( t,y)}{\partial y_{j} \partial y_{k}}.
$$

Adding and subtracting the term
$$
T_{1}=\displaystyle \sum _{k,j=1}^{n}\frac{\partial a_{jk}( t,y)}{\partial y_{j}}\frac{\partial v( t,y)}{\partial y_{k}}
$$
in 
\[
\displaystyle \frac{\partial u(t,x)}{\partial t} - \Delta u( t,x),
\]
we obtain
\begin{eqnarray}
\nonumber \displaystyle \frac{\partial u(t,x)}{\partial t} - \Delta u( t,x) = \frac{\partial v(t,y)}{\partial t}+S_{1}-(S_{2}+S_{3})-T_{1}+T_{1},
\end{eqnarray}
where
$$
S_{1}=\sum _{j=1}^{n}\frac{\partial r_{j}^{-1}( t,r( t,y))}{\partial t}\frac{\partial v( t,y)}{\partial y_{j}},\; S_{2}=\sum _{j=1}^{n}\frac{\partial v( t,y)}{\partial y_{j}} \Delta _{x} r_{j}^{-1}( t,r( t,y)), \;
 S_{3}=\sum _{k,j=1}^{n} a_{jk}( t,y)\frac{\partial ^{2} v( t,y)}{\partial y_{j} \partial y_{k}}.
$$

Applying the divergent identity, we get
$$-(S_{3}+T_{1})=-\sum_{k,j=1}^{n}{\frac{\partial}{\partial y_{j}}\left(a_{jk}(t,y)\frac{\partial v(t,y)}{\partial y_{k}}\right)}.$$

Since $S_{1}$ and $S_{2}$ are sums independent of $a_{jk}$ and $a_{jk}$ is symmetrical, we can change its index to $k$. Let us denote by
\begin{equation}
\label{bk} \displaystyle b_{k}(t,y)=
\frac{\partial r^{-1}_{k}}{\partial t}(t,r(t,y))-\Delta_{x}r^{-1}_{k}(t,r(t,y))+\sum_{j=1}^{n}{\frac{\partial a_{jk}}{\partial y_{j}}}(t,y), \quad \mbox{for $k=1,...,n$,} \end{equation}
then we have the equality

$$S_{1}-S_{2}+T_{1}=\sum_{k=1}^{n}b_{k}(t,y)\frac{\partial v(t,y)}{\partial y_{k}}.$$

Finally, we obtain
$$
\frac{\partial u(t,x)}{\partial t}-\Delta u( t,x) =\frac{\partial v(t,y)}{\partial t}-\sum_{k,j=1}^{n}{\frac{\partial}{\partial y_{j}}\left(a_{jk}(t,y)\frac{\partial v(t,y)}{\partial y_{k}}\right)}+\sum_{k=1}^{n}{b_{k}(t,y)\frac{\partial v(t,y)}{\partial y_{k}}}.
$$

\begin{remark}
By \cite[Proposition IX.6]{Brezis2}, one has that if for some  $t\in(\tau,T)$ the function $u(t,\cdot)\in H^{1}(\mathcal{O}_{t})$, then the function 
$v(t,\cdot)=u(t,r(t,\cdot))\in H^{1}(\mathcal{O}).$ Analogously, if for some $t\in(\tau,T)$ the function $v(t,\cdot) \in H^{1}(\mathcal{O})$, then the function
$u(t,\cdot)= v(t,r^{-1}(t,\cdot))\in H^{1}(\mathcal{O}_{t})$.
\end{remark}

Let us denote by $T(t,y)=G(t,r(t,y))$, where
\begin{equation}
\label{matrix_G}
G(t,x)=\left(\begin{array}{cccc} 
\displaystyle \frac{\partial r^{-1}_{1}}{\partial x_{1}}(t,x)  & ... & \displaystyle\frac{\partial r^{-1}_{n}}{\partial x_{1}}(t,x)\\ \vdots & \vdots  & \vdots\\
\displaystyle \frac{\partial r^{-1}_{1}}{\partial x_{n}}(t,x)  & ... & \displaystyle \frac{\partial r^{-1}_{n}}{\partial x_{n}}(t,x)
\end{array}\right).
\end{equation}
Let  $T^{*}(t,y)$ be the transpose matrix of $T(t,y)$ and $G^{*}(t,x)$ be the  transpose matrix of $G(t,x)$. Note that $T^{*}(t,y)=G^{*}(t,r(t,y)),$ where
$$
G^{*}(t,x)=
\left(\begin{array}{cccc} 
\displaystyle \frac{\partial r^{-1}_{1}}{\partial x_{1}}(t,x)
 &  ... & \displaystyle \frac{\partial r^{-1}_{1}}{\partial x_{n}}(t,x)
 \\ 
\vdots & \vdots   & \vdots\\
 \displaystyle \frac{\partial r^{-1}_{n}}{\partial x_{1}}(t,x)
 &  ... & \displaystyle \frac{\partial r^{-1}_{n}}{\partial x_{n}}(t,x)
  \end{array}\right).
$$
\begin{remark}
Note that
 $G(t,x)$ is the transpose matrix of Jacobian of
   $r^{-1}(t,x)$, that is, $G(t,x)=J^{T}_{x}(r^{-1}(t,x))$, with  $T(t,y)=J^{T}_{x}(r^{-1}(t,r(t,y)))$ and $T^{*}(t,y)=J_{x}(r^{-1}(t,r(t,y)))$. For simplicity, we will use
$T$ and $G$.
\end{remark}

\begin{lemma}
\label{Lema_relacao_normais_Neumman}
Let $\partial \mathcal{O}$ be a surface of class $C^{2}$ and $\partial \mathcal{O}_{t}$ the surface deformed by the diffeomorphism $r(t,\cdot):\overline{\mathcal{O}} \rightarrow \overline{\mathcal{O}}_{t}$
of class $C^{2}$, for all $t \in \mathbb{R}$. Then, 
 $$ n_{t}(x) = G(t,x) n(r^{-1}(t,x)),$$
 where $n_{t}(x)$ and $n(y)$ are the 
 outward normal vectors at $x \in \partial \mathcal{O}_{t}$ and $y \in \partial \mathcal{O}$, respectively.
\end{lemma}
\begin{proof}

Since $\mathcal{O}$ has a $C^{2}$ boundary $\partial \mathcal{O}$, then for each open ball
$B_{r},$ $r>0$, there is a unique function $\psi\in C^{2}(\mathbb{R}^{n},\mathbb{R})$ such that
$\psi(y)=0$, for all $y \in \partial \mathcal{O}\cap B_{r}$, and  $\|\nabla_{y} \psi(y)\|_{\mathbb{R}^n} \equiv 1 \neq 0.$ For more details see \cite{Henry2}. Thus, a candidate for outward normal vector at $\partial \mathcal{O}$  is given by
$$n(y)=\nabla_y \psi(y), \quad \mbox{for all $y \in \partial \mathcal{O}$}.
$$

By other side,
since $r(t,\cdot):\overline{\mathcal{O}} \rightarrow \overline{\mathcal{O}}_{t}$ is a diffeomorphism of class $C^{2}$, for all $t \in \mathbb{R}$, $\psi(y)=\psi(r^{-1}(t,x))=0$ defines $\partial \mathcal{O}_{t}=r(t,\partial \mathcal{O})$ with $\left\|\nabla_{y}\psi(r^{-1}(t,x))\right\|_{\mathbb{R}^{n}} \neq 0.$ Analogously, a candidate for outward normal vector at $\partial \mathcal{O}_{t}$ is given by
$$n_{t}(x)=\nabla_x \psi(r^{-1}(t,x)), \quad \mbox{for all $x \in \partial \mathcal{O}_{t}$}.
$$

Computing the gradient, we obtain
\begin{eqnarray}
\nonumber  n_{t}(x) =\nabla_x \psi(r^{-1}(t,x)) = G(t,x)\nabla_y \psi(r^{-1}(t,x)) = G(t,x)n(r^{-1}(t,x)).
\end{eqnarray}

Therefore,
$$ n_{t}(x) = G(t,x) n(r^{-1}(t,x)).$$
\end{proof}

\begin{remark}\label{Remmm2.4}
With abuse of notation, let us denote 
the unit outward normal vector by
$$
n_{t}(x)=n(t,x)=\frac{1}{\left\|G(t,x)n(r^{-1}(t,x)) \right\|_{\mathbb{R}^{n}}}G(t,x)n(r^{-1}(t,x)), \quad \mbox{for $x \in \partial \mathcal{O}_t$},
$$
or equivalently,
$$n_{t}(r(t,y))=n(t,r(t,y))=\frac{1}{\left\|T(t,y)n(y) \right\|_{\mathbb{R}^{n}}}T(t,y)n(y), \quad \mbox{for $y \in \partial \mathcal{O}$}.
$$
\end{remark}

Now, we use Lemma \ref{Lema_relacao_normais_Neumman}
to compute the relation between the normal derivative in the fixed domain and time-dependent domain. Note that,
\[
\begin{split}
\frac{\partial u}{\partial n_{t}}(t,x)  
&= \displaystyle\langle \nabla _{x} u( t,x) ,n(t,x) \rangle  \\ 
&=\displaystyle\langle \nabla _{x} u( t,x) ,n(t,r(t,y)) \rangle  \\ 
&=\displaystyle \sum _{i=1}^{n}\frac{\partial u( t,x)}{\partial x_{i}} n_{i}(t,r(t,y))  \\
& =\displaystyle \sum _{i=1}^{n}\sum _{k=1}^{n}\frac{\partial v( t,y)}{\partial y_{k}}\frac{\partial r_{k}^{-1}( t,r( t,y))}{\partial x_{i}} n_{i}(t, r( t,y))  \\
&= \displaystyle \sum _{k=1}^{n} \langle \nabla _{x} r_{k}^{-1}( t,r( t,y)) , n( t,x) \rangle \frac{\partial v( t,y)}{\partial y_{k}}.
\end{split}
\]
In other words,
\[
\begin{split}
\frac{\partial u}{\partial n_{t}}(t,x)
&= \displaystyle\sum _{k=1}^{n} \langle \nabla _{x} r_{k}^{-1}( t,r( t,y)) , \frac{1}{\left\|T(t,y)n(y) \right\|_{\mathbb{R}^{n}}}T(t,y) n( y)\rangle \frac{\partial v( t,y)}{\partial y_{k}}  \\
& = \displaystyle\ \frac{1}{\left\|T(t,y)n(y) \right\|_{\mathbb{R}^{n}}}\langle T(t,y) \nabla_y v(t,y) , T(t,y) n( y)\rangle\\
& =\displaystyle\ \frac{1}{\left\|T(t,y)n(y) \right\|_{\mathbb{R}^{n}}}\langle  \nabla_y v(t,y) , T^{*}(t,y)T(t,y) n( y)\rangle,
\end{split}
\]
that is,
\[
\begin{split}
\frac{\partial u}{\partial n_{t}}(t,x)
& =\displaystyle K(t,y)  \sum_{k=1}^{n}\frac{\partial v( t,y)}{\partial y_{k}} \sum _{i,j=1}^{n}\frac{\partial r_{j}^{-1}( t,r( t,y))}{\partial x_i}\frac{\partial r_{k}^{-1}( t,r( t,y))}{\partial x_i}n_{k}(y)\\
& =\displaystyle K(t,y)  \sum_{k=1}^{n} n_{k}(y) \sum _{i,j=1}^{n}\frac{\partial r_{j}^{-1}( t,r( t,y))}{\partial x_i}\frac{\partial r_{k}^{-1}( t,r( t,y))}{\partial x_i}\frac{\partial v( t,y)}{\partial y_{k}},
\end{split}
\]
where $\displaystyle K(t,y)=\frac{1}{\left\|T(t,y)n(y) \right\|_{\mathbb{R}^{n}}}>0$ is a function that depends on diffeomorphims and his derivatives and $n_{k}(y)$ is the $k$-projection of vector $n(y)$ in $k$-th coordinate, $k=1,...,n$.

Using \eqref{ajk}, we can write the representation above as
$$\frac{\partial u}{\partial n_{t}}(t,r^{-1}(t,y))=K(t,y)  \sum_{k=1}^{n} n_{k}(y) \sum _{j=1}^{n} a_{jk}(t,y)\frac{\partial v( t,y)}{\partial y_{k}}.$$

We can rewrite the original problem
(\ref{equacaoproblema}) in the following problem on a fixed domain
\begin{equation}
\label{equacaonodominiofixo}
\left\{
\begin{array}{ll}
\displaystyle \frac{\partial v}{\partial t}-\sum_{k,j=1}^{n}{\frac{\partial}{\partial y_{j}}\left(a_{jk}(t,y)\frac{\partial v}{\partial y_{k}}\right)}+\beta v=f(t,v)-\sum_{k=1}^{n}{b_{k}(t,y)\frac{\partial v}{\partial y_{k}}}, & (t,y)\in(\tau,T]\times \mathcal{O} \\ 
\displaystyle \Gamma(t,v)=0, &  (t,y)\in(\tau,T]\times \partial \mathcal{O} \\
\displaystyle v(\tau,y)=u_{\tau}(r(\tau,y)), & y\in \mathcal{O}
\end{array}
\right.
\end{equation}
where $v=v(t,y)$ and the boundary condition is given by
\begin{equation}
\label{Condicao_Fronteira_Gamma}
\Gamma( t,v)=K(t,y)  \sum_{k=1}^{n} n_{k}(y) \sum _{j=1}^{n} a_{jk}(t,y)\frac{\partial v}{\partial y_{k}}.
\end{equation}
The coefficients
 $a_{jk}(t,y)$ and $b_{k}(t,y)$, $j,k=1,...,n,$ are given by \eqref{ajk} and \eqref{bk}, respectively.

\vspace{0.2cm}

%%%%%%%%%%%%%%%%%%%%%%%%%%%%%%%%%%%%%

\section{Local existence and uniqueness
 of solutions}
\label{sec:problemaabstrato}

\vspace{0.2cm}

In this section, we introduce an abstract formulation
for the problem (\ref{equacaonodominiofixo}), which is on a fixed domain, to obtain the local existence and uniqueness of solutions of this problem. Consequently, we will be obtaining the existence and uniqueness of solutions of the original problem
(\ref{equacaoproblema}) on time-varying domains.

We will divide it into two parts, first a study about the operator and then a study about nonlinearity.

We consider
 the family of unbounded linear operators
 $A(t) : D(A(t)) \subset L^{2}(\mathcal{O}) \rightarrow L^{2}(\mathcal{O}),$ given by
\begin{equation}
\label{operadorabstrato}
 A(t)v=-\sum_{k,j=1}^{n}\frac{\partial}{\partial y_{j}}\left(a_{jk}(t,y)\frac{\partial v}{\partial y_{k}}\right)+\beta v, \quad \mbox{for all } t \in \mathbb{R},
\end{equation}
where $\beta>0$, the coefficient
 $a_{jk}(t,y)$, $j,k=1,...,n$, is given by  (\ref{ajk}) and the domain of
 $A(t)$ is given by
$$D(A(t))=\left\{v \in H^{2}(\mathcal{O}): \;\Gamma(t,v)=0 \; \textrm { in } \;\partial \mathcal{O}\right\},$$
with $\displaystyle \Gamma( t,v)$ given by \eqref{Condicao_Fronteira_Gamma}.

To construct the abstract problem, we will use the theory in
\cite{NolascoNaoAutonomo2,NolascoExistencia, Daners}, which requires in its hypotheses that the domain of the family of operators is independent of
$t \in \mathbb{R}$, at least locally. For this, we will need of hypothesis \textbf{(H1)}.

\begin{lemma}
\label{Dominio_A(t)_fixo}
Suppose that the hypothesis 
\textbf{(H1)} holds. Then, $D(A(t))=D$ for all $t\in I \subset \mathbb{R},$ where
\begin{equation}
\label{Domain_D_Dinfty_fixo}
D=\left\{v \in H^{2}(\mathcal{O}): \sum_{k=1}^{n} n_{k}(y) \sum _{j=1}^{n} \tilde{p}_{jk}(y)\frac{\partial v}{\partial y_{k}}=0 \; \textrm{ in } \; \partial \mathcal{O}\right\},
\end{equation}
with $\displaystyle \tilde{p}_{jk}(y)=\sum_{i=1}^{n}p_{ij}(y)p_{ik}(y)$ and $I\subset \mathbb{R}$ is any bounded interval.
\end{lemma}
\begin{proof}
From \eqref{Condicao_Fronteira_Gamma} we have
$$\displaystyle \Gamma( t,v)=K(t,y)  \sum_{k=1}^{n} n_{k}(y) \sum _{j=1}^{n} a_{j,k}(t,y)\frac{\partial v}{\partial y_{k}}.$$
Using (\ref{ajk}) and \textbf{(H1)}, there are functions $h \in C(\mathbb{R})$,  $p_{ik} \in C^{1}(\mathbb{R}^{n},\mathbb{R})$ and $p_{ij} \in C^{1}(\mathbb{R}^{n},\mathbb{R})$ such that
\begin{equation}
\label{ajk_variavel_separada}
a_{jk}( t,y)=\sum_{i=1}^{n}{\frac{\partial r^{-1}_{k}}{\partial x_{i}}}(t,r(t,y)){\frac{\partial r^{-1}_{j}}{\partial x_{i}}}(t,r(t,y)) =h^{2}(t)\sum _{i=1}^{n}p_{ik}(y)p_{ij}(y)=h^{2}(t)\tilde{p}_{jk}(y), 
\end{equation}
for $t \in I$, $y\in \overline{\mathcal{O}}$ and $j,k=1,...,n$, where $\displaystyle \tilde{p}_{jk}(y)=\sum_{i=1}^{n}p_{ik}(y)p_{ij}(y)$. So,
$$\displaystyle \Gamma( t,v)=K(t,y) h^{2}(t) \sum_{k=1}^{n} n_{k}(y) \sum _{j=1}^{n} \tilde{p}_{jk}(y)\frac{\partial v}{\partial y_{k}}.$$

Since $0<h_0\leq h(t)\leq h_1$  then $h^{2}(t) >0$ for all $t\in \mathbb{R}.$ Hence,  $K(t,y)h^{2}(t)>0$ for all $t \in \mathbb{R}$. Thus,
$$\Gamma(t,v)=0 \Longleftrightarrow  \sum_{k=1}^{n} n_{k}(y) \sum _{j=1}^{n} \tilde{p}_{jk}(y)\frac{\partial v}{\partial y_{k}}=0.
$$
Consequently, $D(A(t))=D$ for all $t\in I.$

\end{proof}

\begin{remark}
From now on, we will use
 $D(A(t))=D$  for all $t \in I,$ where $I \subset \mathbb{R}$ is any bounded interval. From results in \cite{NolascoNaoAutonomo2,NolascoExistencia,Daners}, we need our operator $A(t)$, $t\in I$, or some extension, 
is uniformly sectorial and uniformly Hölder continuous. Let us consider
 $I=[\tau,T]$.
\end{remark}

Analogously to \cite[Lemma 4.2]{PeterKloeden}, we can obtain the following result

\begin{lemma}
\label{lema_elipfort}
Suppose that the hypothesis 
\textbf{(H1)} holds. Then, the operator $A(t)$ is uniformly strongly elliptic
for all $t\in[\tau,T].$
\end{lemma}
\begin{proof}
By hypothesis \textbf{(H1)},  $r\in C^{1}(\mathbb{R}\times\overline{\mathcal{O}},\mathbb{R}^n)$ and $r(t,\cdot):\overline{\mathcal{O}} \rightarrow \overline{\mathcal{O}}_{t}$ is a diffeomorphism of  class $C^{2}$ for all $t \in \mathbb{R}$, then in particular,  $a_{jk} \in  C^{1}([\tau,T]\times\overline{\mathcal{O}})$, $j,k=1,...,n.$

Let $M(t,y)=\left(a_{jk}(t,y)\right)_{n\times n}$ be a matrix  $n\times n$ with the coefficients
$a_{jk}(t,y)$ given by (\ref{ajk}). Then, we can write $M(t,y)=T^{*}(t,y)T(t,y)$ with $T(t,y)=G(t,r(t,y))$, where $G$ is the matrix given by \eqref{matrix_G} and $T^{*}(t,y)$ is 
the transpose matrix of $T(t,y)$.

For any $z \in \mathbb{R}^{n}$, we have
\begin{eqnarray}
\nonumber \sum_{k,j=1}^{n}{a_{jk}(t,y)z_{j}z_{k}} = \langle M(t,y) z, z\rangle= \langle T^{*}(t,y)T(t,y)z , z\rangle 
=\langle T(t,y)z,T(t,y)z\rangle = \left\|T(t,y)z\right\|^{2}_{\mathbb{R}^{n}},
\end{eqnarray}
where $\langle \cdot,\cdot \rangle$ is the usual inner product in $\mathbb{R}^{n}.$

Note that $T(t,y)$ is reversible, so
$$
\left\|T(t,y)z\right\|_{\mathbb{R}^{n}}\geq \left\|T^{-1}(t,y)\right\|_{\mathcal{L}(\mathbb{R}^n)}^{-1}\left\|z\right\|_{\mathbb{R}^{n}}.
$$

Finally, due to the continuity of 
$T^{-1}(t,y)$ and compactness of $[\tau,T]\times \overline{\mathcal{O}}$, we know that $\displaystyle \left\|T^{-1}(t,y)\right\|_{\mathcal{L}(\mathbb{R}^n)}$ is uniformly bounded in $[\tau,T]\times \overline{\mathcal{O}}$. Thus, there exists $\mathrm{c}=\mathrm{c}(r,\tau,T)>0$ such that 
\begin{eqnarray}
\displaystyle \nonumber \sum_{k,j=1}^{n}{a_{jk}(t,y)z_{j}z_{k}} = 
\left\|T(t,y)z\right\|^2_{\mathbb{R}^{n}} 
\displaystyle \nonumber \geq \left\|T^{-1}(t,y)\right\|_{\mathcal{L}(\mathbb{R}^n)}^{-2}\left\|z\right\|^2_{\mathbb{R}^{n}}
\displaystyle \nonumber \geq
\frac{1}{\mathrm{c}^{2}} \left\|z\right\|^{2}_{\mathbb{R}^{n}}
\displaystyle \nonumber = \mathrm{C} \left\|z\right\|^{2}_{\mathbb{R}^{n}},
\end{eqnarray}
for all $(t,y)\in [\tau,T]\times\overline{\mathcal{O}}$, and the result follows.

\end{proof}

With classical arguments, we can prove the following result

\begin{lemma}
Suppose that the hypothesis 
\textbf{(H1)} holds. Then, the operator $A(t)$ is self-adjoint for all $t\in [\tau,T]$.
\end{lemma}
\begin{proof}
Given
 $v,u \in D( A( t))=D$  and $\displaystyle t\in [\tau,T]$, we have
\begin{eqnarray}
\nonumber \displaystyle \langle A(t)v,u\rangle_{L^{2}( \mathcal{O})} & = & \int_{\mathcal{O}} A(t) v(y) u(y) \textrm{d}y\\
\nonumber \displaystyle  & = & -\sum _{k,j=1}^{n}\int_{\mathcal{O}}\left[ \frac{\partial}{\partial y_{j}}\left( a_{jk}(t,y) \frac{\partial v(y)}{\partial y_{k}} \right) u(y) +\beta v(y)  u(y) \right]\textrm{d}y.
\end{eqnarray}

Applying Green's Theorem, we obtain
\begin{eqnarray}
    \nonumber 
\displaystyle &  - &\sum _{k,j=1}^{n}\int _{\mathcal{O}} \frac{\partial}{\partial y_{j}}\left( a_{jk}( t,y) \frac{\partial v(y)}{\partial y_{k}}  \right)u(y) \textrm{d}y  \\
 \displaystyle  &=&  \nonumber \sum _{k,j=1}^{n}\int_{\mathcal{O}}  a_{jk}( t,y) \frac{\partial v(y)}{\partial y_{j}}  \frac{\partial u(y)}{\partial y_{k}}   \textrm{d}y  - \int_{\partial \mathcal{O}}{\frac{1}{K(t,y)}\Gamma(t,v)u(y)\textrm{d}S_{y}}.
\end{eqnarray}

Since  $K(t,y)>0$ and $\displaystyle v \in D$, then $\displaystyle \Gamma ( t,v) =0.$ 

\begin{eqnarray}
\nonumber \displaystyle -\sum_{k,j=1}^{n}\int_{\mathcal{O}} \frac{\partial}{\partial y_{j}} \left( a_{jk}( t,y) \frac{\partial v(y)}{\partial y_{k}} \right)\textrm{d}y    =  \sum _{k,j=1}^{n}\int_{\mathcal{O}}  a_{jk}( t,y) \frac{\partial v(y)}{\partial y_{j}}  \frac{\partial u(y)}{\partial y_{k}}   \textrm{d}y. 
\end{eqnarray}

Applying Green's Theorem again, we obtain
\begin{eqnarray}
\displaystyle   \nonumber \displaystyle & & \sum _{k,j=1}^{n}\int_{\mathcal{O}}  a_{jk}( t,y) \frac{\partial v(y)}{\partial y_{j}}  \frac{\partial u(y)}{\partial y_{k}}   \textrm{d}y \\
\nonumber \displaystyle = &-&\sum _{k,j=1}^{n}\int _{\mathcal{O}} \frac{\partial}{\partial y_{k}}\left( a_{jk}(t,y) \frac{\partial u(y)}{\partial y_{j}}  \right)v(y) \textrm{d}y  +\int_{\partial \mathcal{O}}{\frac{1}{K(t,y)}\Gamma(t,u)v(y)\textrm{d}S_{y}}.
\end{eqnarray}

Since $\displaystyle u \in D$ then  $\Gamma (t,u) =0$. Moreover, $a_{jk}=a_{kj}$, so
\begin{equation}
\nonumber
\displaystyle \sum _{k,j=1}^{n}\int_{\mathcal{O}}  a_{jk}( t,y) \frac{\partial v(y)}{\partial y_{j}}  \frac{\partial u(y)}{\partial y_{k}}   \textrm{d}y 
 = -\sum _{k,j=1}^{n}\int _{\mathcal{O}} \frac{\partial}{\partial y_{k}}\left( a_{kj}(t,y) \frac{\partial u(y)}{\partial y_{j}}  \right)v(y) \textrm{d}y.
\end{equation}

Therefore, 
$\langle A( t) v,u\rangle _{L^{2}( \mathcal{O})}=\langle v,A( t) u\rangle _{L^{2}( \mathcal{O})}$ and consequently $A(t)$ is symmetric. Since $0 \in \rho(A(t))$ for all $t\in [\tau,T]$, then $A(t)$ is self-adjoint.

\end{proof}

\begin{lemma}
\label{Lem_Set_Ext_}
Suppose that the hypothesis 
\textbf{(H1)} holds. Then, the 
operator
$A(t)$ can be extended into a sectorial operator $\tilde{A}(t):H^{1}(\mathcal{O})\subset (H^{1}(\mathcal{O}))'\rightarrow (H^{1}(\mathcal{O}))'$ for all $t\in[\tau,T].$ 
\end{lemma}
\begin{proof}
Let us fix $t \in [\tau,T]$ and consider the following bilinear form
$\displaystyle B_{t}:H^{1}(\mathcal{O}) \times H^{1}(\mathcal{O}) \rightarrow \mathbb{R}$ such that
$$B_{t}(u,v)=\sum_{k,j=1}^{n}{\int_{\mathcal{O}} \left[ a_{jk}(t,y)\frac{\partial u(y)}{\partial y_{j}}\frac{\partial v(y)}{\partial y_{k}}+\beta u(y)v(y) \right]\textrm{d}y}, \quad \mbox{for all } u,v \in H^{1}(\mathcal{O}),$$  where $a_{jk}$ is given by (\ref{ajk}). We will show that $B_{t}$ is a continuous and coercive bilinear form.

By Lemma \ref{lema_elipfort}, we have
\begin{eqnarray*}
\nonumber\displaystyle B_{t}(u,u)  \displaystyle & = & \sum_{k,j=1}^{n}{\int_{\mathcal{O}} \left[ a_{jk}(t,y)\frac{\partial u(y)}{\partial y_{j}} \frac{\partial u(y)}{\partial y_{k}}+\beta u^{2}(y)\right]\textrm{d}y} \\
\nonumber &\geq & {\int_{\mathcal{O}} \left[ C\left\|\nabla u(y)\right\|_{\mathbb{R}^n}^2+\beta u^{2}(y)\right] \textrm{d}y}   
 \geq \delta\left\| u\right\|^{2}_{H^{1}(\mathcal{O})},
\end{eqnarray*}
where $\displaystyle \delta=\min\left\{C,\beta\right\}$ is independent of $t$.

Since $r\in C^{1}(\mathbb{R}\times\overline{\mathcal{O}},\mathbb{R}^{n})$ and 
$r(t,\cdot):\overline{\mathcal{O}} \rightarrow \overline{\mathcal{O}}_{t}$ is a diffeomorphism of  class $C^{2}$ for all $t \in \mathbb{R}$, then $a_{jk} \in  C^{1}([\tau,T]\times\overline{\mathcal{O}})$, $j,k=1,...,n.$ In particular, there exists $c_{jk}>0$ independent of $t$ such that $\left\|a_{jk}(t,\cdot)\right\|_{L^{\infty}(\overline{\mathcal{O}})}\leq c_{jk}$ for all $t\in [\tau,T]$. Thus,
\begin{eqnarray*}
\displaystyle |B_{t}(u,v)| 
     & \leq &   \sum_{k,j=1}^{n}{\int_{\mathcal{O}} \left| a_{jk}(t,y)\frac{\partial u(y)}{\partial y_{j}} \frac{\partial v(y)}{\partial y_{k}}+\beta u(y)v(y)\right|\textrm{d}y}\\
     & \leq &  \sum_{k,j=1}^{n} \left\|a_{jk}(t,\cdot)\right\|_{L^{\infty}(\overline{\mathcal{O}})}{\int_{\mathcal{O}} \left| \frac{\partial u(y)}{\partial y_{j}} \frac{\partial v(y)}{\partial y_{k}}\right|\textrm{d}y}+ \beta\int_{\mathcal{O}} \left| u(y) v(y)\right|\textrm{d}y  \\
     & \leq &  \sum_{k,j=1}^{n} c_{jk}\left\|\frac{\partial u}{\partial y_{j}}\right\|_{L^{2}(\mathcal{O})}\left\|\frac{\partial v}{\partial y_{k}}\right\|_{L^{2}(\mathcal{O})}+\beta\left\|u\right\|_{L^{2}(\mathcal{O})}\left\|v\right\|_{L^{2}(\mathcal{O})} \\
     &\leq& M\left\|u\right\|_{H^{1}(\mathcal{O})}\left\|v\right\|_{H^{1}(\mathcal{O})},
\end{eqnarray*}
where $\displaystyle M=\max_{k,j=1,..,n}\left\{c_{jk},\beta \right\}>0$ is independent of $t$.

Therefore, $B_{t}$ is a continuous and coercive bilinear form defined in $H^{1}(\mathcal{O})$, and by \cite[Theorem 2.1]{Yagi} the associated unbounded linear operators (their parts) are sectoral operators in
 $H^{1}(\mathcal{O})$, $L^{2}(\mathcal{O})$  and $(H^{1}(\mathcal{O}))'$ with constant $\displaystyle \frac{M+\delta}{\delta}$. 

Now, given $u,v \in D$ and applying
Green's Theorem, we obtain
\begin{eqnarray}
\nonumber B_{t}(u,v)=-\sum_{k,j=1}^{n}\int_{\mathcal{O}} \left[\frac{\partial}{\partial y_{k}}\left(a_{jk}(t,y)\frac{\partial u(y)}{\partial y_{j}}\right)+\beta u(y) \right] v(y) \textrm{d}y+{\int_{\partial \mathcal{O}} \frac{1}{K(t,y)}\Gamma(t,u)v(y)\textrm{d}S_{y}}.
\end{eqnarray}
Since $\Gamma(t,u)=0$, we have
$$B_{t}(u,v)=-\sum_{k,j=1}^{n}{\int_{\mathcal{O}} \left[\frac{\partial}{\partial y_{k}}\left(a_{jk}(t,y)\frac{\partial u(y)}{\partial y_{j}}\right)+\beta u(y)\right]v(y)\textrm{d}y}, \quad \mbox{for all } u,v\in D.$$

Finally, if there exists $A(t)$ such that $\langle A(t)u,v\rangle =B_{t}(u,v),$ then it is clear that
$$A(t)u=-\sum_{k,j=1}^{n}\frac{\partial}{\partial y_{k}}\left(a_{jk}(t,y)\frac{\partial u(y)}{\partial y_{j}}\right)+\beta u(y), \quad \mbox{for all } u \in D.$$

\end{proof}

Next, we obtain that the operator is uniformly sectorial and uniformly Hölder continuous.

\begin{theorem}
\label{operadorverificado}
Suppose that the hypothesis 
\textbf{(H1)} holds. Then, there exists $k>0$ such that 
the operator
$\tilde{A}(t)+kI:H^{1}(\mathcal{O})\subset (H^{1}(\mathcal{O}))'\rightarrow (H^{1}(\mathcal{O}))'$ is uniformly sectorial and uniformly Hölder continuous in $\mathcal{L}(H^{1}(\mathcal{O}),(H^{1}(\mathcal{O}))')$ for all $t\in[\tau,T]$.
\end{theorem}
\begin{proof}
Let $u\in H^{1}(\mathcal{O})$ and $\lambda \in \Sigma_{0,\phi}$. Using  Lemma \ref{Lem_Set_Ext_}, there exist constant $M$ and $\delta$ independent of $t\in [\tau, T]$ such that
\begin{eqnarray}
\nonumber \left\|u\right\|_{H^{1}(\mathcal{O})} = \left\| [\lambda I-(-\tilde{A}(t))]^{-1}[\lambda I-(-\tilde{A}(t))]u \right\|_{H^{1}(\mathcal{O})} 
\nonumber \leq \frac{M+\delta}{\delta|\lambda|} \left\|(\lambda I +\tilde{A}(t))u \right\|_{(H^{1}(\mathcal{O}))'}.
\end{eqnarray}

Choosing
 $k>1$ we get $k+\lambda$ is an element of
 $\Sigma_{0,\phi}$, so 
\begin{eqnarray}
\nonumber \left\|u\right\|_{H^{1}(\mathcal{O})}&\leq & \frac{M+\delta}{\delta|\lambda+k|} \left\|[(\lambda +k)I -(-\tilde{A}(t))]u \right\|_{(H^{1}(\mathcal{O}))'}
\nonumber \\
\nonumber &\leq& \frac{M+\delta}{\delta(|\lambda|+1)} \left\|[(\lambda +k)I -(-\tilde{A}(t))]u \right\|_{(H^{1}(\mathcal{O}))'}.
\end{eqnarray} 

Since $\lambda+k\in \rho(-A(t))$ then $\lambda \in\rho(-(A(t)+k))$ and there exists $[\lambda I -(-\tilde{A}(t)-kI)]^{-1}$, for all $\lambda \in \Sigma_{0,\phi}$,
such that
$$\left\|[\lambda I -(-\tilde{A}(t)-kI)]^{-1}\right\|_{\mathcal{L}((H^1(\mathcal{O}))',H^1(\mathcal{O}))}\leq \frac{M+\delta}{\delta(|\lambda|+1)}, $$
that is, the operator $\tilde{A}(t)+kI$ is uniformly sectorial for all
 $t\in [\tau,T]$.

Finally, we will show that the operator is uniformly H\"older continuous. Note that as a consequence of the hypothesis
 \textbf{(H1)}, there exists a constant $c>0$ such that
$$
|h^2(s)-h^2(t)|=|(h(s)-h(t))(h(s)+h(t))|\leq c |s-t|^{\theta}, \quad \mbox{for all } t,s\in [\tau,T],
$$
for some $\theta \in (0,1]$. Thus, for each $j,k=1,..,n$, there exists  $c_{jk}>0$ independent of
 $t$ such that
\begin{equation}
\label{holderajk}
\left\| a_{jk}(s,\cdot)-a_{jk}(t,\cdot) \right\|_{L^{\infty}(\overline{\mathcal{O}})}\leq c_{jk}|s-t|^{\theta}, \quad \mbox{for all } t,s\in [\tau,T],
\end{equation}
where $a_{jk}$ is given by \eqref{ajk_variavel_separada}.

Let $v \in D$, $\phi \in H^{1}(\mathcal{O})$ and $t,s\in [\tau,T]$,
using (\ref{holderajk}) we have
\begin{eqnarray}
\nonumber \displaystyle \displaystyle \left|\int_{\mathcal{O}}{[A(t)v(y)-A(s)v(y)]\phi\textrm{d}y}\right| & \leq  & \int_{\mathcal{O}}{\sum_{k,j=1}^{n}{\left| [a_{jk}(s,y)-a_{jk}(t,y)]\frac{\partial v}{\partial y_{k}}\frac{\partial \phi}{\partial y_{j}}\right|}\textrm{d}y} \\
\nonumber \displaystyle  &\leq & \sum_{k,j=1}^{n} {\left\| a_{jk}(s,\cdot)-a_{jk}(t,\cdot)\right\|_{L^{\infty}(\overline{\mathcal{O}})} \left\| \frac{\partial v}{\partial y_{k}}\right\|_{L^{2}(\mathcal{O})} \left\|\frac{\partial \phi}{\partial y_{j}}\right\|_{L^{2}(\mathcal{O})}}  \\
\nonumber \displaystyle  &\leq & 
c|t-s|^{\theta} \left\| v\right\|_{H^{1}(\mathcal{O})} \left\|\phi\right\|_{H^{1}(\mathcal{O})},
\end{eqnarray}
where $c=n \displaystyle \max_{k,j=1,..,n}\{c_{jk}\}>0.$

Taking $u \in H^1(\mathcal{O})$ and using the density of $D$ in $H^1(\mathcal{O})$, we obtain
\begin{eqnarray}
\left\|[\tilde{A}(t)-\tilde{A}(s)]u\right\|_{(H^1(\mathcal{O}))'}   \leq  c |t-s|^{\theta}\left\|u\right\|_{H^1(\mathcal{O})}.
\end{eqnarray}
Thus, taking the supreme with $\left\|u\right\|_{H^1(\mathcal{O})}=1$, we conclude that
$$\left\|[\tilde{A}(t)-\tilde{A}(s)]\right\|_{\mathcal{L}(H^1(\mathcal{O}),(H^1(\mathcal{O}))')} \leq  c |t-s|^{\theta}.$$
\end{proof}

\begin{remark}
It is also possible to prove that
 the operator
$\tilde{A}(t)+kI$ is uniformly sectorial and uniformly Hölder continuous in $\mathcal{L}(L^{2}(\mathcal{O}))$ for all $t\in[\tau,T]$.
 
\end{remark}

\begin{theorem}
\label{A_composta_inversaLimitado}
Suppose that the hypothesis 
\textbf{(H1)} holds and $\beta>0$ is large sufficiently. Then, the operators $A(t)A^{-1}(s)$ and $\overline{A(t)A^{-1}(s)}$ are uniformly bounded in
$\mathcal{L}(L^{2}(\mathcal{O}))$ and $\mathcal{L}((H^{1}(\mathcal{O}))')$ for all $t,s\in[\tau,T]$, respectively.
\end{theorem}
\begin{proof}
By Theorem \ref{operadorverificado} there exist positive constants
 $M$ and $\delta$ independent of $s\in[\tau,T]$ such that
$$\left\|\overline{A^{-1}(s)}\right\|_{\mathcal{L}((H^{1}(\mathcal{O}))',H^{1}(\mathcal{O}))}\leq \frac{M+\delta}{\delta}.$$

Moreover, $\overline{A(t)}$ is also an isomorphism of
 $H^{1}(\mathcal{O})$ in $(H^{1}(\mathcal{O}))'$, then there exists $C>0$ independent of $t\in[\tau,T]$ such that
$$\left\|\overline{A(t)}\right\|_{\mathcal{L}(H^{1}(\mathcal{O}),(H^{1}(\mathcal{O}))')} \leq C.$$

Thus, there exists $L>0$ independent of $t,s \in [\tau,T]$ such that
$$\left\|\overline{A(t)A^{-1}(s)}\right\|_{\mathcal{L}((H^{1}(\mathcal{O}))')} \leq L.$$

Since $A(t)$ is a  densely defined operator and $\overline{A(t)}$ is its unique extension, by density, there is
$\tilde{L}>0$ independent of $t,s \in [\tau,T] $ such that
$$
\left\|A(t)A^{-1}(s)\right\|_{\mathcal{L}(L^{2}(\mathcal{O}))}\leq \tilde{L}.
$$

\end{proof}

\begin{remark}
We will assume $\beta>0$ is large sufficiently and, by abuse of notation, we will denote the extension of $A(t)$ and all its parts, that is, its realizations, by $A(t).$
\end{remark}

Since $A(t)$ is self-adjoint and positive, then it has bounded purely imaginary powers (see \cite{Yagi}), and by Theorem \ref{A_composta_inversaLimitado}, $A(t)A^{-1}(s)$ and $\overline{A(t)A^{-1}(s)}$ are uniformly bounded. Thus, from \cite[Corollary
 3.5, Proposition
 3.6 and Corollary
 3.7]{NolascoNaoAutonomo2},
 we can construct a scale of Banach spaces that are independent of $t$.

Consider the operator family
$A(t) : D(A(t)) \subset L^{2}(\mathcal{O}) \rightarrow L^{2}(\mathcal{O})$ given by \eqref{operadorabstrato} and let $E^{0}=L^{2}(\mathcal{O})$ and $E^{1}=D(A(t))=D$, where the domain $D$ is independent of $t$ and it is given by
\eqref{Domain_D_Dinfty_fixo}. There is a scale of fractional power spaces
 $E^{\alpha}$, $\alpha \in \mathbb{R}$, associated to $A(t)$, which satisfies
$$E^{\alpha} \hookrightarrow H^{2\alpha}(\mathcal{O}), \quad \mbox{for $\alpha \geq 0$},
$$
with $E^{\frac{1}{2}}=H^{1}(\mathcal{O})$.

Since the family of operators
$\{A(t): t\in [\tau,T]\}$ is already in the conditions of \cite[Lemmas 16.1, 16.2 and 16.7]{Daners}, then to obtain the local existence and uniqueness of solutions to the problem
 (\ref{equacaonodominiofixo}), we will need to study the nonlinearity and its respective abstract map.

Given $ \displaystyle \frac{1}{2} \leq \alpha < 1$ we define the following map
\begin{equation}
\label{nao_linearidade_abstrata_F}
\begin{array}{cccl}
  F:   & \mathbb{R} \times E^{\alpha} & \rightarrow &E^{0}\\
     &(t,v)& \mapsto & F(t,v)=f^{e}(t,v) - \langle \vec{b}(t,\cdot),  \nabla v \rangle
\end{array}
\end{equation}
where $f^e:\mathbb{R} \times E^{\alpha} \rightarrow E^{0}$ 
is the Nemitski\u{\i} operator associated to $f$, given by $f^e(t,v)(y):=f(t,v(y))$ for any $(t,v)\in \mathbb{R}\times E^{\alpha}$ and $y\in \mathcal{O}$,  $\vec{b}(t,y)=(b_1(t,y),...,b_n(t,y))\in \mathbb{R}^n$, with each coefficient
 $b_k(t,y)$, $k=1,...,n$,  given in \eqref{bk} and
\begin{equation}\label{bk_Abstract_Novo}
\langle \vec{b}(t,y),  \nabla v(y) \rangle = \sum_{k=1}^{n}{b_{k}(t,y)\frac{\partial v(y)}{\partial y_{k}}}, \quad \mbox{for all $(t,y)\in \mathbb{R}\times \mathcal{O}$}.
 \end{equation}

With the abstract operator given by \eqref{operadorabstrato} and the abstract nonlinearity defined in \eqref{nao_linearidade_abstrata_F}, for $\displaystyle \frac{1}{2} \leq \alpha <1$, we can write (\ref{equacaonodominiofixo}) as the following abstract parabolic equation
\begin{equation}
\label{equacaonoabstratanofixo}
\left\{
\begin{array}{lr}
\dot{v}(t)+A(t)v(t)=F(t,v(t)), & t> \tau \\ 
\displaystyle v(\tau)=v_{\tau} \in E^{\alpha}.
\end{array}
\right.
\end{equation}

Under the above conditions and from
\cite[Teorema 6.1]{Pazy}, there is a unique evolution process $U(t,\tau)$ for homogeneous problem
\begin{equation}
\label{equacaonoabstratanofixohomogeneo_S}
\left\{
\begin{array}{lr}
\dot{v}(t)+A(t)v(t)=0, & t>\tau \\ 
\displaystyle v(\tau)=v_{\tau} \in E^{\alpha},
\end{array}
\right.
\end{equation}
which is associated to the abstract parabolic equation
 \eqref{equacaonoabstratanofixo},  satisfying $\left\|U(t,\tau)\right\|_{\mathcal{L}(E^{\alpha})}\leq C,$ for all $\tau \leq t \leq T$, for some constant $C>0$. 

Now, to obtain the existence and uniqueness of solutions of the abstract parabolic problem \eqref{equacaonoabstratanofixo},we need the following result

\begin{lemma}
\label{naolinearidadef}
Suppose that the hypotheses \textbf{(H1)} and \textbf{(H2)} hold
and let $ \displaystyle \frac{1}{2} \leq\alpha  < \min \left\{ 1,\frac{n}{4}\right\}$ and  $I\subset \mathbb{R}$ be bounded interval. Then, there exists $c>0$ independent of $t$ such that   
$$
\left\|F(t,v)-F(t,w)\right\|_{E^{0}}  \leq  c \left\| v-w \right\|_{E^{\alpha}}(1+\left\| v\right\|^{\rho}_{E^{\alpha}}+\left\| w\right\|^{\rho}_{E^{\alpha}}),
$$
for all $t\in I$ and $v,w\in E^{\alpha}$. In particular, for any $r>0$, there exists a constant $L(r)>0$ such that
\begin{equation}
\label{Lipschitz_Local_F}
\left\|F(t,v)-F(t,w)\right\|_{E^{0}}  \leq  L(r) \left\| v-w \right\|_{E^{\alpha}},
\end{equation}
for all $t\in I$ and $v,w\in E^\alpha$ with $\left\|v\right\|_{E^\alpha}, \left\|w\right\|_{E^\alpha}<r,$ that is, the map $F:(t,\cdot): E^{\alpha}  \rightarrow E^{0}$ is locally
 Lipschitz, uniformly in $t\in I$.
\end{lemma}
\begin{proof}
Given $t\in I$ and  $v,w\in E^\alpha$, using \eqref{cres02} and the Hölder and Minkowski inequalities, we obtain
\begin{eqnarray}
\nonumber \displaystyle  & &  \left\|f^e(t,v)-f^e(t,w)\right\|_{E^{0}}  \\
\nonumber \displaystyle & = &
\left\|f^e(t,v)-f^e(t,w)\right\|_{L^2(\mathcal{O})} \\
\nonumber \displaystyle &\leq &  c\left\| |v-w|(1+|v|^{\rho} + |w|^{\rho})\right\|_{L^{2}(\mathcal{O})}\\
\nonumber \displaystyle & \leq &  c\left\| v-w\right\|_{L^{2q}(\mathcal{O})}\left(\left\|1+|v|^{\rho} + |w|^{\rho}\right\|_{L^{2q'}(\mathcal{O})}\right) \\
\nonumber \displaystyle & \leq &  c\left\| v-w\right\|_{L^{2q}(\mathcal{O})}\left(|\mathcal{O}|^{\frac{1}{2q'}}+\left\||v|^{\rho}\right\|_{L^{2q'}(\mathcal{O})} + \left\||w|^{\rho}\right\|_{L^{2q'}(\mathcal{O})}\right) \\
\nonumber \displaystyle 
&\leq &  \tilde{c}\left\| v-w\right\|_{L^{2q}(\mathcal{O})}\left(1+\left\|v\right\|^{\rho}_{L^{2q'\rho}(\mathcal{O})} + \left\|w\right\|^{\rho}_{L^{2q'\rho}(\mathcal{O})}\right),
\end{eqnarray}
where $\displaystyle \frac{1}{q}+\frac{1}{q'}=1$ and $|\mathcal{O}|$ is the measure of
 $\mathcal{O}$. For example, taking $ \displaystyle q = \frac{n}{n-4\alpha}$ and $\displaystyle q'=\frac{n}{4\alpha}$ with $\displaystyle 0<\rho \leq \frac{4\alpha}{n-4\alpha}$ and $\displaystyle 0\leq\alpha  < \min \left\{ 1,\frac{n}{4}\right\}$, we have
\begin{equation}
%\label{imersao1}
\nonumber
E^{\alpha} \hookrightarrow H^{2\alpha}(\mathcal{O}) \hookrightarrow L^{2q}(\mathcal{O})    \quad \mbox{and} \quad E^{\alpha} \hookrightarrow H^{2\alpha}(\mathcal{O}) \hookrightarrow L^{2q'\rho}(\mathcal{O}). \end{equation}

Thus, there exists a constant $c_1>0$ independent of $t$ such that
\begin{equation}
\label{Lips_1_f}
\left\|f^e(t,v)-f^e(t,w)\right\|_{E^{0}}  \leq  c_1 \left\| v-w \right\|_{E^{\alpha}}(1+\left\| v\right\|^{\rho}_{E^{\alpha}}+\left\| w\right\|^{\rho}_{E^{\alpha}}),
\end{equation}
for all $t\in I$ and $v,w\in E^{\alpha}$.

Since $r\in C^{1}(\mathbb{R}\times\overline{\mathcal{O}},\mathbb{R}^n)$ and $r(t,\cdot):\overline{\mathcal{O}} \rightarrow \overline{\mathcal{O}}_{t}$ is a diffeomorphism of  class $C^2$ for all $t\in\mathbb{R}$, 
then $b_{k} \in  C(\mathbb{R}\times\overline{\mathcal{O}})$, $k=1,...,n$. In particular, there exists $c_k>0$ independent of $t$ such that
$$
\left\|b_{k}(t,\cdot)\right\|_{L^{\infty}(\overline{\mathcal{O}})}\leq c_k, \quad \mbox{for all $t\in I$}.
$$
On the other hand, given $t\in I$ and  $v,w\in E^\alpha$ and using (\ref{bk_Abstract_Novo}), we have
\begin{eqnarray}
\nonumber
&& \|\langle \vec{b}(t,\cdot),  \nabla (v-w) \rangle
\|_{E^0}\\
& = & \nonumber \|\langle \vec{b}(t,\cdot),  \nabla (v-w) \rangle
\|_{L^{2}(\mathcal{O})}\\
& = & \nonumber\left[\int_{\mathcal{O}}|\langle \vec{b}(t,y),  \nabla (v(y)-w(y)) \rangle|^2\textrm{d}y \right]^\frac{1}{2}\\
\nonumber &\leq & 
\left[\int_{\mathcal{O}} \sum_{k=1}^{n}  \left|b_{k}(t,y)\frac{\partial}{\partial y_{k}}(v(y)-w(y)) \right|^{2} \textrm{d}y \right]^\frac{1}{2}\\
\nonumber &\leq& \sum_{k=1}^{n} \left\|b_{k}(t,\cdot)\right\|_{L^{\infty}(\overline{\mathcal{O}})} \left\|\frac{\partial}{\partial y_{k}}(v-w) \right\|_{L^{2}(\mathcal{O})}\\
\nonumber & \leq & c_2\left\|v-w \right\|_{E^{\frac{1}{2}}}.
\end{eqnarray}
where $c_2=n\displaystyle \max_{k=1,...,n}\{c_k\}>0$ is independent of $t$.

Now, for $\alpha\geq  \displaystyle \frac{1}{2}$, since $E^{\alpha} \hookrightarrow E^{\frac{1}{2}}$ then there exists $c_3>0$ independent of $t$ such that
\begin{equation}
\label{Lips_3_f}
 \|\langle \vec{b}(t,\cdot),  \nabla (v-w) \rangle
\|_{E^0} \leq c_3 \left\|v-w \right\|_{E^{\alpha}},
\end{equation}
for all $t\in I$ and $v,w\in E^{\alpha}$.

Therefore, by \eqref{Lips_1_f} and \eqref{Lips_3_f}, there exists $c>0$ independent of $t$ such that
\begin{equation}
\label{Lips_4_f}
\left\|F(t,v)-F(t,w)\right\|_{E^{0}}  \leq  c \left\| v-w \right\|_{E^{\alpha}}(1+\left\| v\right\|^{\rho}_{E^{\alpha}}+\left\| w\right\|^{\rho}_{E^{\alpha}}),
\end{equation}
for all $t\in I$ and $v,w\in E^{\alpha}$.

In particular, if $v,w\in E^\alpha$ satisfy
$\left\|v\right\|_{E^\alpha}, \left\|w\right\|_{E^\alpha}<r,$
then \eqref{Lipschitz_Local_F} follows from \eqref{Lips_4_f}.

\end{proof}

Finally, we can get the problem
\eqref{equacaonoabstratanofixo} is well posed in $E^{\alpha}$, for $ \displaystyle \frac{1}{2} \leq\alpha  < \min \left\{ 1,\frac{n}{4}\right\}$.

\begin{theorem}
\label{existencialocal}
Suppose that the hypotheses \textbf{(H1)} and \textbf{(H2)} hold
and let $\displaystyle \frac{1}{2} \leq\alpha  < \min \left\{ 1,\frac{n}{4}\right\}$ be given. 
If $v_{\tau} \in E^\alpha$ with
$\left\|v_{\tau}\right\|_{E^{\alpha}}\leq r$ for some $r>0$, then there  exists a number
$T_{max}=T_{max}(\tau,\alpha, r)>\tau$ such that the problem (\ref{equacaonoabstratanofixo}) has a unique maximal mild  solution $v(\cdot,\tau;v_\tau):[\tau,T_{max})\rightarrow E^{\alpha}.$

Given $0\leq \beta\leq \alpha<1$, there exists a constant $C>0$ such that for any $x,y\in E^{\alpha}$, satisfying $\left\|x\right\|_{E^{\alpha}},\left\|y\right\|_{E^{\alpha}}\leq r$, we have 
$$\left\|v(t,\tau;x)-v(t,\tau;y)\right\|_{E^{\alpha}}\leq L(t-\tau)^{\beta-\alpha}\left\|x-y\right\|_{E^\beta},\quad \mbox{for all $t\in \tilde{I}$},
$$
where $\tilde{I}\subset [\tau,T_{max}(\tau,x))\cap[\tau,T_{max}(\tau,y)).
$
\end{theorem}
\begin{proof}
The result follows from Theorem
 \ref{operadorverificado}, Lemma \ref{naolinearidadef} and \cite[Lemmas 16.1, 16.2 and 16.7]{Daners}.

\end{proof}

%%%%%%%%%%%%%%%%%%%%%%%%%%%%%%%%%%%%

\vspace{0.2cm}

\section{Global existence}
\label{sec:existenciaglobal}

\vspace{0.2cm}

To study the existence of pullback attractors of a problem, as well as its asymptotic behavior, we need to ensure that the local solution of (\ref{equacaonoabstratanofixo}) is globally defined. For this, we will need of a sign condition on nonlinearity given by \textbf{(H3)} and  of \cite[Corollary 16.3]{Daners}.

This way, we obtain an estimate for our abstract nonlinear application
$F:\mathbb{R} \times E^{\alpha}  \rightarrow E^{0}$ defined  by \eqref{nao_linearidade_abstrata_F}.

\begin{lemma}
\label{limitacaoaplicacaoabstrata}
Suppose that the hypotheses \textbf{(H1)}, \textbf{(H2)} and \textbf{(H3)} hold
and let $ \displaystyle \frac{1}{2} \leq\alpha  < \min \left\{ 1,\frac{n}{4}\right\}$ and   $I\subset \mathbb{R}$ be a bounded interval. Then, there exist $K_1,K_2>0$   independent of $t$ such that
$$
\left\|F(t,v)\right\|_{E^{0}}\leq K_{1} \left\|v\right\|_{E^{\alpha}}+K_{2}, \quad \mbox{for all $t \in I$ and $v\in E^{\alpha}$}.
$$
In particular, the map $F:\mathbb{R} \times E^{\alpha}  \rightarrow E^{0}$
 takes bounded subsets of $E^{\alpha}$ into  bounded subsets of $E^{0}$.
\end{lemma}

\begin{proof}
Given $t\in I$ and  $v,w\in E^\alpha$, using the Minkowski inequality, we obtain
$$
\nonumber \displaystyle   \left\|f^e(t,v)\right\|_{E^{0}}  \nonumber \displaystyle   =
\nonumber \left\|f^e(t,v)\right\|_{L^2(\mathcal{O})} 
\nonumber  \leq   \left\| k_1|v|+k_2\right\|_{L^{2}(\mathcal{O})}
\nonumber \displaystyle  \leq   k_1\left\| v\right\|_{L^{2}(\mathcal{O})}+k_2|\mathcal{O}|^{\frac{1}{2}} 
\nonumber \displaystyle  \leq \tilde{k}_1\left\| v\right\|_{E^{\frac{1}{2}}}+\tilde{k}_2,
$$
where  $|\mathcal{O}|$ is the measure of
$\mathcal{O}$ and $\tilde{k}_1,\tilde{k}_2>0$ do not depend of
 $t$.

Now, for $\alpha > \displaystyle \frac{1}{2}$, since $E^{\alpha} \hookrightarrow E^{\frac{1}{2}}$ then there exists $\tilde{k}_3>0$ independent of $t$ such that
\begin{equation}
\label{sinal_f_2}
\left\|f^e(t,v)\right\|_{E^{0}} \leq \tilde{k}_3\left\| v\right\|_{E^{\alpha}}+\tilde{k}_2, \quad \mbox{for all $t\in I$ and $v\in E^{\alpha}$}.
\end{equation}

Moreover, by  \eqref{Lips_3_f}, there exist $\tilde{k}_4>0$ independent of $t$ such that
\begin{equation}
\label{sinal_b_1}
 \|\langle \vec{b}(t,\cdot),  \nabla v \rangle
\|_{E^0} \leq \tilde{k}_4 \left\|v \right\|_{E^{\alpha}}, \quad \mbox{for all $t\in I$ and  $v\in E^{\alpha}$}.
\end{equation}
Thus, the result
follows from
  \eqref{sinal_f_2} and \eqref{sinal_b_1}. 

\end{proof}

As a consequence of this result, we obtain the global existence of solution.

\begin{theorem}
\label{existenciaglobaldesolucoes}
Suppose that the hypotheses \textbf{(H1)}, \textbf{(H2)} and \textbf{(H3)} hold and let $ \displaystyle \frac{1}{2} \leq\alpha  < \min \left\{ 1,\frac{n}{4}\right\}$ be given. If $v_{\tau} \in E^\alpha$ with
$\left\|v_{\tau}\right\|_{E^{\alpha}}\leq r$ for some $r>0$, then the unique  mild solution $v(t,\tau;v_\tau) \in E^{\alpha}$ of problem (\ref{equacaonoabstratanofixo}) exists globally.
\end{theorem}

\begin{proof}
We just need to show that the solution does not explode in finite time intervals.

From Theorem \ref{existencialocal}, we know that there is a unique mild solution for problem (\ref{equacaonoabstratanofixo}) given by variation of constants formula
$$
v(t,\tau;v_\tau)=U(t,\tau)v_{\tau}+\int_{\tau}^{t}U(t,s)F(s,v(s,\tau;v_{\tau}))\textrm{d}s,\quad t\in [\tau,T_{max}),
$$
with initial data
$v(\tau,\tau;v_\tau)=v_\tau$, where $U(t,\tau)$ is the evolution process associated to the homogeneous problem
\eqref{equacaonoabstratanofixohomogeneo_S}. For simplicity, we will simply use $v(t):=v(t,\tau;v_\tau)$.

So, for $t\in (\tau,T]$ and  $v_{\tau} \in E^\alpha$, with
$\left\|v_{\tau}\right\|_{E^{\alpha}}\leq r$ for some $r>0$, using Lemma \ref{limitacaoaplicacaoabstrata} we have
\begin{eqnarray}
\label{eq_Gronnwall_norma_finita_1}
\nonumber \displaystyle \left\|v(t)\right\|_{E^\alpha} &\leq& \left\|U(t,\tau)v_{\tau}\right\|_{E^\alpha}+\int_{\tau}^{t}\left\|U(t,s)F(s,v(s))\right\|_{E^\alpha} \textrm{d}s \\
\nonumber \displaystyle &\leq& M_{\alpha}(t-\tau)^{-\alpha}\left\|v_{\tau}\right\|_{E^\alpha}+M_{\alpha}\int_{\tau}^{t}(t-s)^{-\alpha}\left\|F(s,v(s))\right\|_{E^0} \textrm{d}s \\
\nonumber \displaystyle &\leq& M_{\alpha}(t-\tau)^{-\alpha}\left\|v_{\tau}\right\|_{E^\alpha}+M_{\alpha}\int_{\tau}^{t}(t-s)^{-\alpha}(K_{1}\left\|v(s)\right\|_{E^\alpha}+K_2) \textrm{d}s \\
\nonumber \displaystyle &=& M_{\alpha}(t-\tau)^{-\alpha}\left\|v_{\tau}\right\|_{E^\alpha}+M_{\alpha}K_{1}\int_{\tau}^{t}(t-s)^{-\alpha}\left\|v(s)\right\|_{E^\alpha} \textrm{d}s+ \frac{M_{\alpha}K_2(T-\tau)}{1-\alpha}(t-\tau)^{-\alpha}
 \\ 
 \displaystyle & = & \displaystyle M_{\alpha}\left(\left\|v_{\tau}\right\|_{E^\alpha}+\frac{K_2(T-\tau)}{1-\alpha}\right) (t-\tau)^{-\alpha} +  M_{\alpha}K_{1}\int_{\tau}^{t}(t-s)^{-\alpha}\left\|v(s)\right\|_{E^\alpha} \textrm{d}s. 
\end{eqnarray}

Applying Gronwall's inequality in
\eqref{eq_Gronnwall_norma_finita_1} (see \cite[Corollary 16.6]{Daners}), we get there exists a constant
$c(\alpha,M_\alpha,K_1,T)>0$ such that
 $$
\left\|v(t)\right\|_{E^\alpha} \leq  \displaystyle M_{\alpha}\left(\left\|v_{\tau}\right\|_{E^\alpha}+\frac{K_2(T-\tau)}{1-\alpha}\right)c(\alpha,M_\alpha,K_1,T) (t-\tau)^{-\alpha}, \quad \mbox{for all $t\in (\tau,T]$}.
$$
Therefore, given $r>0$, there exists a constant $C=C(r)>0$ such that
$$
\sup\{ \|v(t)\|_{E^{\alpha}}:\; \|v_\tau\|_{E^{\alpha}}\leq r,\; \tau \in\mathbb{R} \;\text{ and }\; t\in [\tau, T_{max})\}\leq C.
$$
With this we conclude that
 $ \left\|v(t)\right\|_{E^\alpha}$ is bounded in finite time intervals. Consequently, by \cite[Corollary 16.3]{Daners} we have that the mild solution of (\ref{equacaonoabstratanofixo}) can be extended globally to all $t\geq \tau.$

\end{proof}

\begin{remark}
The Theorem \ref{existenciaglobaldesolucoes} 
implies that for each
$\tau\in\mathbb{R}$ and $v_{\tau}\in E^{\alpha}$, the solution $v(t,\tau;v_\tau)$ of (\ref{equacaonoabstratanofixo}), with $v(\tau,\tau;v_\tau)=v_\tau$, is defined for all
$t\geq \tau$, that is,
$T_{max}=+\infty$. Consequently, we can define a process
  $\{S(t,\tau): t\geq \tau\}$ in $E^{\alpha}$ by $S(t,\tau)v_{\tau}= v(t,\tau;v_\tau)$,
$t\geq \tau$. Moreover, each bounded subset of  $E^{\alpha}$ has bounded orbit and pullback orbit.
\end{remark}
%%%%%%%%%%%%%%%%%%%%%%%%%%%%
\vspace{0.2cm}

\section{Existence of pullback attractors}
\label{Sec_principal_Atrator_Pullback}

\vspace{0.2cm}

In this section, we will prove that the problem (\ref{equacaonoabstratanofixo}) has a pullback attractor in $E^{\frac{1}{2}}$. We consider the family of operators $\left\{A(t):t\in\mathbb{R}\right\}$ defined in \eqref{operadorabstrato} with domain $D(A(t))=D$, where $D$ is given by \eqref{Domain_D_Dinfty_fixo}, uniformly sectorial and uniformly Hölder continuous for all $t\in \mathbb{R}$, with the operator
$A(t)A(\tau)^{-1}$ uniformly bounded for 
$\tau,t\in \mathbb{R}$. 

To study the pullback dynamics of the parabolic evolution equation
(\ref{equacaonoabstratanofixo}), it will be necessary 
that the solution of the homogeneous evolution equation (\ref{equacaonoabstratanofixohomogeneo_S}) has exponential decay to zero. For this, we will use the asymptotic behavior of diffeomorphism given by hypothesis \textbf{(H4)}.

Initially, we note that by hypothesis
  \textbf{(H1)} and \eqref{ajk_variavel_separada},  we can write
$$a_{jk}( t,y) =h^{2}(t) \tilde{p}_{jk}(y),
$$
for all $t \in \mathbb{R}$, $y\in \overline{\mathcal{O}}$ and $j,k=1,...,n$, where $\displaystyle \tilde{p}_{jk}(y)=\sum_{i=1}^{n}p_{ik}(y)p_{ij}(y).$ Moreover, just as in the proof of Lemma \ref{lema_elipfort}, we can use the above notation to write the following matrix
$n\times n$,  
$$
M(t,y)=\left(a_{jk}(t,y)\right)_{n\times n}=h^{2}(t)\left(\tilde{p}_{jk}(y)\right)_{n\times n}=T^{*}(t,y)T(t,y),
$$
with $T(t,y)=G(t,r(t,y))$, where $G$ is the matrix given by \eqref{matrix_G} and $T^{*}(t,y)$ is 
the transpose matrix of $T(t,y)$.

Now, by hypothesis
\textbf{(H4)} we have
$$
|h^{2}(t)-h^{2}(\tau)|\to 0, \quad \mbox{as $t-\tau \to \infty$}.
$$
Thus, looking for each projection of
 $M(t,y)$, we obtain 
$$
|a_{jk}(t,y)-a_{jk}(\tau,y)|\to 0 \quad \mbox{and} \quad \left|\frac{\partial}{\partial y_{j}}\left(a_{jk}(t,y)-a_{jk}(\tau,y)\right)\right|\to 0, \quad \mbox{as $t-\tau \to \infty$,}
$$
for all $y\in \overline{\mathcal{O}}$. Consequently,
\begin{equation}
\label{coeficientes_a_infty_jk}
\|a_{jk}(t,\cdot)-a_{jk}(\tau,\cdot)\|_{C^{1}(\overline{\mathcal{O}})}\to 0, \quad \mbox{as $t-\tau \to \infty$.}
\end{equation}
With this convergence, we can prove the following result

\begin{lemma}
\label{convergenciaainfinito}
Suppose that the hypotheses \textbf{(H1)} and \textbf{(H4)} hold. Then, 
$$\left\|( A( t) -A( \tau )) A^{-1}( r)\right\|_{\mathcal{L}(L^{2}(\mathcal{O}))}\to 0, \quad \mbox{as $t-\tau \to \infty$,}
$$
uniformly for all $r \in\mathbb{R}$.
\end{lemma}
\begin{proof}
For $\tau, t \in \mathbb{R}$ and $v\in D$.
Using Green's Theorem and Cauchy-Schwarz inequality, we have 
\begin{eqnarray}
\nonumber && \left\|( A( t) -A( \tau ))v \right\|^{2}_{L^{2}(\mathcal{O})}   
\\
\nonumber &= &  \int_{\mathcal{O}}\left|\sum_{k,j=1}^{n} \left[-\frac{\partial}{\partial y_{j}}\left(a_{jk}(t,y)\frac{\partial v}{\partial  y_{k}}\right)+\frac{\partial}{\partial y_{j}}\left(a_{jk}(\tau,y)\frac{\partial v}{\partial y_{k}}\right)\right]\right|^2 \textrm{d}y \\
\nonumber&\leq&  \sum_{k,j=1}^{n}\int_{\mathcal{O}} \left|\frac{\partial}{\partial y_{j}}\left[\left(a_{jk}(\tau,y)-a_{jk}(t,y)\right)\frac{\partial v}{\partial  y_{k}}\right]\right|^2 \textrm{d}y \\
\nonumber&\leq&  \sum_{k,j=1}^{n}\left[\int_{\mathcal{O}} \left|\frac{\partial}{\partial y_{j}}\left(a_{jk}(\tau,y)-a_{jk}(t,y)\right)\right|^2\left|\frac{\partial v}{\partial  y_{k}}\right|^2 \textrm{d}y + \int_{\mathcal{O}} \left|a_{jk}(\tau,y)-a_{jk}(t,y)\right|^2\left|
\frac{\partial^2 v}{\partial  y_{j}\partial y_{k}}\right|^2 \textrm{d}y \right]\\
\nonumber&\leq & \sum_{k,j=1}^{n} \left[ \left\|\frac{\partial}{\partial y_{j}} \left(a_{jk}(t,\cdot)-a_{jk}(\tau,\cdot)\right)\right\|^{2}_{L^{\infty}(\overline{\mathcal{O}})}\left\|
\frac{\partial v}{\partial y_{k}}\right\|^{2}_{L^{2}(\mathcal{O})}\right. \\
\nonumber &+& \left.
\|a_{jk}(t,\cdot)-a_{jk}(\tau,\cdot)\|^{2}_{L^{\infty}(\overline{\mathcal{O}})} \left\|\frac{\partial^2 v}{\partial  y_{j}\partial y_{k}}
\right\|^{2}_{L^{2}(\mathcal{O})}
\right]\\
\nonumber&\leq & \sum_{k,j=1}^{n} \|a_{jk}(t,\cdot)-a_{jk}(\tau,\cdot)\|^{2}_{C^{1}(\overline{\mathcal{O}})}\left\| v\right\|^{2}_{H^{2}(\mathcal{O})}.
\end{eqnarray}
Therefore,
$$
\left\|( A( t) -A( \tau ))v \right\|^{2}_{L^{2}(\mathcal{O})}  \leq\sum_{k,j=1}^{n} \|a_{jk}(t,\cdot)-a_{jk}(\tau,\cdot)\|^{2}_{C^{1}(\overline{\mathcal{O}})}\left\| v\right\|^{2}_{H^{2}(\mathcal{O})}, \quad \mbox{for all $v\in D$}.
$$
Replacing $v$ by $A^{-1}(r)\tilde{v}$, we have
$$
\left\|( A( t) -A( \tau )) A^{-1}(r)\tilde{v}\right\|^{2}_{L^{2}(\mathcal{O})}  \leq\sum_{k,j=1}^{n} \|a_{jk}(t,\cdot)-a_{jk}(\tau,\cdot)\|^{2}_{C^{1}(\overline{\mathcal{O}})}\left\| \tilde{v}\right\|^{2}_{L^{2}(\mathcal{O})}, \quad \mbox{for all $\tilde{v}\in L^{2}(\mathcal{O})$}.
$$

Now, taking the square root on both sides and using \eqref{coeficientes_a_infty_jk}, we obtain
$$
\|( A( t) -A( \tau )) A^{-1}(r)\|_{\mathcal{L}(L^{2}(\mathcal{O}))}\to 0, \quad\mbox{as $t-\tau \to \infty$}.
$$

\end{proof}

Using Lemma \ref{convergenciaainfinito} and adapting the result  in \cite[Theorem 8.1]{Pazy}, we conclude that the evolution process $U(t,\tau)$ of (\ref{equacaonoabstratanofixohomogeneo_S}) has exponential decay in $L^2(\mathcal{O})$.

\begin{lemma}
\label{decaimento_exponencial_do_processo}
Suppose that the hypotheses \textbf{(H1)} and \textbf{(H4)} hold. Then, there exist $K>0$ and $b>0$ independent of $t$ and $\tau$ such that
$$\left\| U(t,\tau) \right\|_{\mathcal{L}(L^2(\mathcal{O}))}\leq Ke^{-b(t-\tau)}, \quad  \mbox{for all $t\geq \tau$},
$$
where $U(t,\tau)$ is the linear evolution process associated with the  problem (\ref{equacaonoabstratanofixohomogeneo_S}).
\end{lemma}
\begin{proof}
We note first that since the operator $A(t)$ is uniformly sectorial for all $t\in\mathbb{R},$ then $-A(\tau)$ is the infinitesimal generator of an analytic semigroup $\{e^{-sA(\tau)} \in \mathcal{L}(L^2(\mathcal{O})):\; s\geq 0\}$. With this and of the fact that $0 \in \rho(A(t))$, follows that there are constants $\nu_{0}>0$ and $M>0$ independent of $t,\tau\in \mathbb{R}$ such that
 \begin{eqnarray}
\label{semigrupo_1} \left\|e^{-sA(\tau)}\right\|_{\mathcal{L}(L^2(\mathcal{O}))} \leq Me^{-\nu_{0}s}, \quad \mbox{for $s\geq 0$},\\
\label{operadorsemigrupo_1}\left\|A(\tau)e^{-sA(\tau)}\right\|_{\mathcal{L}(L^2(\mathcal{O}))} \leq Ms^{-1}e^{-\nu_{0}s}, \quad \mbox{for $s>0$}.
\end{eqnarray}

We define 
\begin{equation}
\label{Def_Rho_Pazy}
\eta(\mu)=\sup_{r\in \mathbb{R}}\{ \left\|(A(t)-A(\tau))A(r)^{-1}\right\|_{\mathcal{L}(L^2(\mathcal{O}))}:\; 0\leq \mu \leq t-\tau\}.
\end{equation}

By hypothesis and Lemma \ref{convergenciaainfinito}, $\eta(\mu)$ is finite for $\mu \geq 0$ and
$$\eta(\mu) \to 0, \quad \textrm{as $\mu \to \infty$}.
$$ 

Combining (\ref{Def_Rho_Pazy}) and  the fact that  $A(t)$ is uniformly Hölder continuous for all $t\in \mathbb{R}$, there are $C>0$ and $\theta \in (0,1]$ independent of $t$, $\tau$ and $r$ such that
\begin{eqnarray}
\nonumber \displaystyle & & \left\|(A(t)-A(\tau))A(r)^{-1}\right\|^{2}_{\mathcal{L}(L^2(\mathcal{O}))} \\
\nonumber  &=& \displaystyle \left\|(A(t)-A(\tau))A(r)^{-1}\right\|_{\mathcal{L}(L^2(\mathcal{O}))} \left\|(A(t)-A(\tau))A(r)^{-1}\right\|_{\mathcal{L}(L^2(\mathcal{O}))} \\
\nonumber  & \leq & \displaystyle  \eta(\mu)C|t-\tau|^{\theta}, \quad \mbox{for $\mu \leq t-\tau$}. 
\end{eqnarray}
Thus,
\begin{equation}
\label{desigualdeR_1}
\left\|(A(t)-A(\tau))A(r)^{-1}\right\|_{\mathcal{L}(L^2(\mathcal{O}))} \leq \tilde{c} (\eta(\mu)|t-\tau|^{\theta})^{\frac{1}{2}}=\tilde{c}\sqrt{\eta(\mu)}|t-\tau|^{\frac{\theta}{2}},  \quad \mbox{for $\mu \leq t-\tau$},  \end{equation}
and for some $\tilde{c}>0$ independent of $t$, $\tau$ and $r$, with $r\in \mathbb{R}$.

We note that $U(t,\tau)$ is given by
$$U(t,\tau)=e^{-(t-\tau)A(\tau)}+\int_{\tau}^{t}U(t,z)\left[A(\tau)-A(z)\right]e^{-(z-\tau)A(\tau)}\textrm{d}z.$$
Moreover, using the construction in
 \cite[Section 5.6]{Pazy}, we can write 
$$\int_{\tau}^{t}U(t,z)\left[A(\tau)-A(z)\right]e^{-(z-\tau)A(\tau)}dz=\int_{\tau}^{t}e^{-(t-z)A(z)}R(z,\tau)\textrm{d}z,$$
where $\displaystyle R(z,\tau)=\sum_{m=1}^{\infty}R_{m}(z,\tau)$, with the sequence of operators $R_{m}$ defined inductively by
\begin{equation}
\nonumber
\left\{
\begin{array}{lll}
\displaystyle R_{1}(t,\tau)=(A(\tau)-A(t))e^{-(t-\tau)A(\tau)}\\
\displaystyle R_{m}(t,\tau)= \displaystyle \int_{\tau}^{t}R_{1}(t,z)R_{m-1}(z,\tau)\textrm{d}z, & m\geq 2.
\end{array}
\right.
\end{equation}
Consequently,
\begin{equation}
\label{equacaointegralabstratanaoautonoma_PAZY}
U(t,\tau)=e^{-(t-\tau)A(\tau)}+\int_{\tau}^{t}e^{-(t-z)A(z)}R(z,\tau)\textrm{d}z.
\end{equation}

We note
\begin{eqnarray}
\nonumber \displaystyle \left\|R_{1}(t,\tau)\right\|_{\mathcal{L}(L^2(\mathcal{O}))} &=&  \left\|(A(\tau)-A(t))e^{-(t-\tau)A(\tau)}\right\|_{\mathcal{L}(L^2(\mathcal{O}))}\\ 
\nonumber \displaystyle &=&  \left\|(A(\tau)-A(t))A(\tau)^{-1}A(\tau)e^{-(t-\tau)A(\tau)}\right\|_{\mathcal{L}(L^2(\mathcal{O}))}\\
\nonumber \displaystyle &\leq&   \left\|(A(\tau)-A(t))A(\tau)^{-1}\right\|_{\mathcal{L}(L^2(\mathcal{O}))}\left\|A(\tau)e^{-(t-\tau)A(\tau)}\right\|_{\mathcal{L}(L^2(\mathcal{O}))} \\
\nonumber \displaystyle &=&   \left\|(A(t)-A(\tau))A(\tau)^{-1}\right\|_{\mathcal{L}(L^2(\mathcal{O}))}\left\|A(\tau)e^{-(t-\tau)A(\tau)}\right\|_{\mathcal{L}(L^2(\mathcal{O}))}. 
\end{eqnarray}

Using (\ref{operadorsemigrupo_1}) and (\ref{desigualdeR_1}), we have
\begin{eqnarray}
%\label{desigualdefinalR_1}
\nonumber
    \left\|R_{1}(t,\tau)\right\|_{\mathcal{L}(L^2(\mathcal{O}))} \leq \displaystyle \tilde{c}\sqrt{\eta(\mu)}(t-\tau)^{\frac{\theta}{2}} \frac{M}{t-\tau}e^{-\nu_{0}(t-\tau)} 
 \nonumber
   = \displaystyle C\sqrt{\eta(\mu)}(t-\tau)^{\frac{\theta}{2}-1} e^{-\nu_{0}(t-\tau)}, 
\end{eqnarray}
for $0\leq \mu \leq  t-\tau$, and by induction on $m\in\mathbb{N}$, it follows that
\begin{eqnarray}
\nonumber \displaystyle \left\|R_{m}(t,\tau)\right\|_{\mathcal{L}(L^2(\mathcal{O}))} = \left\|\int_{\tau}^{t}R_{1}(t,z)R_{m-1}(z,\tau)dz\right\|_{\mathcal{L}(L^2(\mathcal{O}))}
\nonumber \displaystyle \leq \frac{e^{-\nu_{0}(t-\tau)}}{t-\tau}\frac{(C\sqrt{\eta(\mu)}(t-\tau)^{\frac{\theta}{2}})^{m}}{\Gamma(\frac{m\theta}{2})}.
\end{eqnarray}

For $\beta \in (0,1]$ there is a constant $C_\beta>0$ such that
\begin{equation}
\label{seriepotenciagamme}
\sum_{m=1}^{\infty} \frac{x^{m}}{\Gamma(m\beta)}\leq C_{\beta}xe^{2x^{\frac{1}{\beta}}}, \quad \mbox{for all  $x \geq 0$}.
\end{equation}

Since $\displaystyle \theta \in (0,1]$  then $\frac{\theta}{2} \in \left(0,\frac{1}{2}\right]$ and $C\sqrt{\eta(\mu)}(t-\tau)^{\frac{\theta}{2}}\geq 0$ for $t \geq \tau.$ Thus, using (\ref{seriepotenciagamme}),  for $0\leq \mu \leq t- \tau$, we have
\begin{eqnarray}
 \nonumber \displaystyle \sum_{m=1}^{\infty}\left\| R_{m}(t,\tau)\right\|_{\mathcal{L}(L^2(\mathcal{O}))}&\leq & \sum_{m=1}^{\infty}\frac{e^{-\nu_{0}(t-\tau)}}{t-\tau}\frac{(C\sqrt{\eta(\mu)}(t-\tau)^{\frac{\theta}{2}})^{m}}{\Gamma(\frac{m\theta}{2})}  \\
 \nonumber \displaystyle &\leq & \displaystyle \frac{e^{-\nu_{0}(t-\tau)}}{t-\tau}\sum_{m=1}^{\infty}\frac{(C\sqrt{\eta(\mu)}(t-\tau)^{\frac{\theta}{2}})^{m}}{\Gamma(\frac{m\theta}{2})}  \\
 \nonumber \displaystyle &\leq & \frac{e^{-\nu_{0}(t-\tau)}}{t-\tau}C_{\theta}C\sqrt{\eta(\mu)}(t-\tau)^{\frac{\theta}{2}}e^{2({C\sqrt{\eta(\mu)}(t-\tau)^{\frac{\theta}{2}})^{\frac{2}{\theta}}}}\\
 \nonumber \displaystyle &\leq & \displaystyle \frac{e^{-\nu_{0}(t-\tau)}}{t-\tau}C_{\theta}C\sqrt{\eta(\mu)}(t-\tau)^{\frac{\theta}{2}}e^{2({C^{2}\eta(\mu))^{\frac{1}{\theta}}(t-\tau)}}\\ 
\nonumber \displaystyle &\leq & \displaystyle C_{\theta}C\sqrt{\eta(\mu)}(t-\tau)^{\frac{\theta}{2}-1}e^{2({C^{2}\eta(\mu))^{\frac{1}{\theta}}(t-\tau)}-\nu_{0}(t-\tau)}\\
\nonumber \displaystyle &\leq & \displaystyle  C_{\theta} C \sqrt{\eta(\mu)}(t-\tau)^{\frac{\theta}{2}-1}e^{-\left[\nu_{0}-2(C^{2}\eta(\mu))^{\frac{1}{\theta}}\right](t-\tau)}.
\end{eqnarray}

Since $\eta(\mu)\to 0$ as $\mu\to \infty$, then there exist $\mu_{0}>0$ such that for $\mu_0 \leq \mu \leq t-\tau$, we have $\nu=\nu_{0}-2(C^{2}\eta(\mu))^{\frac{1}{\theta}}>0$ and 
\begin{equation}
\label{equacaointegralabstratanaoautonoma_PAZY_2}
\left\| R(t,\tau) \right\|_{\mathcal{L}(L^2(\mathcal{O}))}\leq \sum_{m=1}^{\infty}\left\| R_{m}(t,\tau)\right\|_{\mathcal{L}(L^2(\mathcal{O}))} \leq C_{\theta} C \sqrt{\eta(\mu)}(t-\tau)^{\frac{\theta}{2}-1}e^{-\nu(t-\tau)}.
\end{equation}
Using (\ref{semigrupo_1}) and (\ref{equacaointegralabstratanaoautonoma_PAZY_2}), we have
\begin{eqnarray}
\nonumber\displaystyle \left\|\int_{\tau}^{t}e^{-(t-z)A(z)}R(z,\tau)\textrm{d}z\right\|_{\mathcal{L}(L^2(\mathcal{O}))} &\leq & \int_{\tau}^{t}\left\|e^{-(t-z)A(z)}R(z,\tau)\right\|_{\mathcal{L}(L^2(\mathcal{O}))}\textrm{d}z \\
\nonumber \displaystyle &\leq & \int_{\tau}^{t}\left\|e^{-(t-z)A(z)}\right\|_{\mathcal{L}(L^2(\mathcal{O}))} \left\|R(z,\tau)\right\|_{\mathcal{L}(L^2(\mathcal{O}))}\textrm{d}z \\
\nonumber \displaystyle &\leq & MC_{\theta} C \sqrt{\eta(\mu)}\int_{\tau}^{t}e^{-\nu_{0}(t-z)} (z-\tau)^{\frac{\theta}{2}-1}e^{-\nu(z-\tau)}\textrm{d}z\\
\nonumber \displaystyle &\leq & MC_{\theta} C \sqrt{\eta(\mu)}\int_{\tau}^{t}e^{-b(t-z)} (z-\tau)^{\frac{\theta}{2}-1}e^{-b(z-\tau)}\textrm{d}z\\
\nonumber \displaystyle &\leq & MC_{\theta} C  \sqrt{\eta(\mu)}\int_{\tau}^{t} (z-\tau)^{\frac{\theta}{2}-1}e^{-b(t-\tau)}\textrm{d}z\\
\nonumber \displaystyle &= & MC_{\theta} C  \sqrt{\eta(\mu)}e^{-b(t-\tau)}\int_{\tau}^{t} (z-\tau)^{\frac{\theta}{2}-1}\textrm{d}z \\
\displaystyle &= & \tilde{K} \sqrt{\eta(\mu)}e^{-b(t-\tau)} (t-\tau)^{\frac{\theta}{2}},
\label{equacaointegralabstratanaoautonoma_PAZY_1} 
\end{eqnarray}
where $\tilde{K}= \frac{2}{\theta}MC_{\theta} C>0$ and $b = \min\{\nu_{0},\nu\}>0$ are independent of $t$ and $\tau$.

So, from  (\ref{semigrupo_1}), 
(\ref{equacaointegralabstratanaoautonoma_PAZY}) and (\ref{equacaointegralabstratanaoautonoma_PAZY_1}), we deduce
\begin{eqnarray}
\nonumber \displaystyle   \left\|U(t,\tau)\right\|_{\mathcal{L}(L^2(\mathcal{O}))} &\leq &\left\| e^{-(t-\tau)A(\tau)}\right\|_{\mathcal{L}(L^2(\mathcal{O}))}+\left\|\int_{\tau}^{t}e^{-(t-z)A(z)}R(z,\tau)\textrm{d}z\right\|_{\mathcal{L}(L^2(\mathcal{O}))}\\
\nonumber \displaystyle &\leq& Me^{-\nu_{0}(t-\tau)}+ \tilde{K} \sqrt{\eta(\mu)}e^{-b(t-\tau)} (t-\tau)^{\frac{\theta}{2}}\\
\nonumber \displaystyle &\leq& e^{-b(t-\tau)}\left[ M+ \tilde{K} \sqrt{\eta(\mu)} (t-\tau)^{\frac{\theta}{2}}\right].
\end{eqnarray}

Finally, using again
 $\eta(\mu)\to 0$ as $\mu\to \infty$ (so $t-\tau \to \infty$), there is a constant $K>0$ independent of $t$ and $\tau$ such that
$$\left\|U(t,\tau)\right\|_{\mathcal{L}(L^2(\mathcal{O}))} \leq Ke^{-b(t-\tau)},\quad t\geq \tau.
$$
\end{proof}

Now, we conclude that
 the evolution process $U(t,\tau)$ of (\ref{equacaonoabstratanofixohomogeneo_S}) has exponential decay in $E^{\alpha}$, for $\displaystyle  \frac{1}{2}\leq\alpha<1$.

\begin{lemma}
\label{decaimentoexponencial}
Suppose that the hypotheses \textbf{(H1)} and \textbf{(H4)} hold. Then, there exist $b>0$, $M>0$ and $M_\alpha>0$ such that the linear evolution process $U(t,\tau)$ associated with the homogeneous equation \eqref{equacaonoabstratanofixohomogeneo_S} satisfies
\begin{equation}
\label{decaimentoexponencialdoprocesso}
\left\|U(t,\tau)w\right\|_{E^{\alpha}} \leq Me^{-b(t-\tau)}\left\|w\right\|_{E^{\alpha}}, \quad  t\geq \tau,  
 \end{equation}
for all  $w\in E^{\alpha}$ and, for each bounded time interval $I=[-T,T]$ ($T>0$),
\begin{equation}
\label{decaimentoexponencialdoprocesso2}
\displaystyle 
\left\|U(t,\tau)w\right\|_{E^{\alpha}} \leq M_{\alpha}(t-\tau)^{-\alpha}e^{-b(t-\tau)}\left\|w\right\|_{E^{0}}, %\quad  t> \tau,
\end{equation}
for all  $w\in E^{0}$ and for all  $ -T\leqslant \tau< t\leqslant T$.
\end{lemma}

\begin{proof}
From Theorem \ref{A_composta_inversaLimitado} we have $A(t)A(\tau)^{-1}$ is uniformly bounded, for  $\tau,t\in \mathbb{R}$. Hence and Lemma \ref{convergenciaainfinito}, we can apply Lemma \ref{decaimento_exponencial_do_processo} and get the inequality
(\ref{decaimentoexponencialdoprocesso}).

Now, for each bounded time interval $I=[-T,T]$ ($T>0$), 
(\ref{decaimentoexponencialdoprocesso2}) follows from Theorem \ref{operadorverificado} jointly with the fact that
\[
\left\|A^{\alpha}(\xi)
U(t,\tau)w\right\|_{E^{0}} \leq M K(\alpha)(t-\tau)^{-\alpha}e^{-b(t-\tau)}\left\|w\right\|_{E^{0}},\quad \mbox{for all $-T\leqslant \tau< t\leqslant T$},
\]
according to  \cite[Proposition 3.11]{NolascoNaoAutonomo2} and 
\cite[Lemma 2.9]{NolascoExistencia}.
\end{proof}

Now, we can prove that the process $\{S(t,\tau): t\geq \tau\}$ generated by (\ref{equacaonoabstratanofixo}) is strongly pullback bounded dissipative in the sense of \cite[Definition 2.22]{NolascoAtratores}.

\begin{lemma}
\label{limitacaotinfinitodassoluções}
Suppose that the hypotheses \textbf{(H1)}, \textbf{(H2)}, \textbf{(H3)} and \textbf{(H4)} hold and let $ \displaystyle \frac{1}{2} \leq\alpha  < \min \left\{ 1,\frac{n}{4}\right\}$ be given. If $v_{\tau} \in E^\alpha$ with
$\left\|v_{\tau}\right\|_{E^{\alpha}}\leq r$ for some $r>0$, then there exists a constant $R>0$ independent of $t$ such that
$$
\lim_{\tau \rightarrow - \infty}\left\|v(t,\tau;v_\tau)\right\|_{E^{\alpha}} 
 \leq R, %M_\alpha K_2b^{\alpha-1}\Gamma(1-\alpha) e^{M_\alpha K_1b^{\alpha-1}\Gamma(1-\alpha)}, \quad \frac{1}{2}\leq \alpha <1
$$
where $v(t,\tau;v_\tau)$ is the mild solution of  (\ref{equacaonoabstratanofixo}). Consequently, the process
 $\{S(t,\tau): t\geq \tau\}$ 
 is strongly pullback bounded dissipative.
\end{lemma}

\begin{proof}
Initially, from Theorems \ref{existencialocal} and \ref{existenciaglobaldesolucoes}, we have that there exist a unique global solution  for (\ref{equacaonoabstratanofixo}) given by variation of constants formula
$$
v(t,\tau;v_\tau)=U(t,\tau)v_{\tau}+\int_{\tau}^{t}U(t,s)F(s,v(s,\tau;v_{\tau}))\textrm{d}s,\quad t\geq \tau,
$$
with initial data $v(\tau,\tau;v_\tau)=v_\tau$, where $U(t,\tau)$ is the evolution process associated to the homogeneous problem (\ref{equacaonoabstratanofixohomogeneo_S}). For simplicity, we will simply use $v(t):=v(t,\tau;v_\tau)$.

Taking $t\geq  \tau$ and $v_{\tau} \in E^\alpha$, with
$\left\|v_{\tau}\right\|_{E^{\alpha}}\leq r$ for some $r>0$, and using Lemmas \ref{decaimentoexponencial} and \ref{limitacaoaplicacaoabstrata}, we have
\begin{eqnarray}
\label{eq_Gronnwall_norma_finita_2_N} 
\nonumber \displaystyle \left\|v(t)\right\|_{E^\alpha}
\nonumber \displaystyle &\leq &\left\|U(t,\tau)v_{\tau}\right\|_{E^\alpha}+\int_{\tau}^{t}\left\|U(t,s)F(s,v(s))\right\|_{E^\alpha} \textrm{d}s \\
\nonumber \displaystyle &\leq & M e^{-b(t-\tau)}\left\|v_{\tau}\right\|_{E^\alpha}+M_{\alpha}\int_{\tau}^{t}(t-s)^{-\alpha}e^{-b(t-s)}\left\|F(s,v(s))\right\|_{E^0}  \textrm{d}s \\
\nonumber \displaystyle &\leq & Me^{-b(t-\tau)}\left\|v_{\tau}\right\|_{E^\alpha}+M_{\alpha}\int_{\tau}^{t} (t-s)^{-\alpha}e^{-b(t-s)} (K_{1}\left\|v(s)\right\|_{E^\alpha}+K_2) \textrm{d}s \\
\nonumber \displaystyle &= & Me^{-b(t-\tau)}\left\|v_{\tau}\right\|_{E^\alpha}+M_{\alpha}K_1\int_{\tau}^{t} (t-s)^{-\alpha}e^{-b(t-s)} \left\|v(s)\right\|_{E^\alpha} \textrm{d}s\\
\nonumber \displaystyle &+ & M_{\alpha}K_2\int_{\tau}^{t} (t-s)^{-\alpha}e^{-b(t-s)} \textrm{d}s\\
\displaystyle &\leq & Me^{-b(t-\tau)}\left\|v_{\tau}\right\|_{E^\alpha}+M_{\alpha}K_2 b^{\alpha-1}\Gamma(1-\alpha)
\displaystyle +M_{\alpha}K_1\int_{\tau}^{t} (t-s)^{-\alpha}e^{-b(t-s)} \left\|v(s)\right\|_{E^\alpha} \textrm{d}s,
\end{eqnarray}
where $\displaystyle \Gamma(z)=\int^{\infty}_{0} \xi^{z-1}e^{-\xi} \textrm{d}\xi$, with  $\textrm{Re}(z)>0$, is the  Gamma function. Applying Gronwall's inequality in
\eqref{eq_Gronnwall_norma_finita_2_N} (see \cite[Lemma 6.23]{NolascoAtratores}), we have
\[
 \begin{split}
\nonumber \displaystyle  \left\|v(t)\right\|_{E^\alpha} &\leq \displaystyle [Me^{-b(t-\tau)}\left\|v_{\tau}\right\|_{E^\alpha}+M_\alpha K_2b^{\alpha-1}\Gamma(1-\alpha)]e^{M_{\alpha}K_1\int_{\tau}^{t} (t-s)^{-\alpha}e^{-b(t-s)} \textrm{d}s}\\
\nonumber \displaystyle 
&\leq [Mre^{-b(t-\tau)}+M_\alpha K_2b^{\alpha-1}\Gamma(1-\alpha)]e^{M_\alpha K_1b^{\alpha-1}\Gamma(1-\alpha)}.
\end{split}
\]

Since $b>0$ and $\tau \to -\infty$ implies that
 $t-\tau \to \infty$, we obtain
$$
\lim_{\tau \rightarrow - \infty}\left\|v(t,\tau;v_\tau)\right\|_{E^{\alpha}}=\displaystyle \lim_{t-\tau \rightarrow  \infty}\left\|v(t,\tau;v_\tau)\right\|_{E^{\alpha}} 
\leq \displaystyle R,
$$
where $R=M_\alpha K_2b^{\alpha-1}\Gamma(1-\alpha) e^{M_\alpha K_1b^{\alpha-1}\Gamma(1-\alpha)}>0$ is independent of $t$.

Consequently, this proves that for any bounded set
 $B\subset E^{\alpha}$, there exist a constant $R>0$ and a time
 $\tau_0(B)$ such that
\begin{equation}
\label{desigualdade_S_fortemente_limitado_dissipativo}
S(t,\tau)B\subset B_{E^{\alpha}}(0,R),\quad \mbox{for all $\tau\leq \tau_0(B)$ and $t \geq \tau_{0}(B)$},
\end{equation}
uniformly for all $t\in \mathbb{R}$, where $B_{E^{\alpha}}(0,R)$ is the ball closed in $E^{\alpha}$ with center at the origin and radius $R$, that is, $\{S(t,\tau): t\geq \tau\}$  is strongly pullback bounded dissipative.
\end{proof}

We have shown that the process generated by our problem (\ref{equacaonoabstratanofixo})  is strongly pullback bounded dissipative, but to apply \cite[Teorema 2.23]{NolascoAtratores} and deduce that the problem (\ref{equacaonoabstratanofixo}) has a pullback attractor in $E^{\frac{1}{2}}$ we need to show that $\{S(t,\tau): t\geq \tau\}$ is also pullback asymptotically compact in the sense of \cite[Definition 2.8]{NolascoAtratores}.

Consider the following decomposition of
 $S(t,\tau)$,
$$S(t,\tau)=U(t,\tau) + L(t,\tau) ,$$
where $U(t,\tau)$ is the solution operator for the homogeneous problem
(\ref{equacaonoabstratanofixohomogeneo_S}) and  
\begin{equation}
\label{Lcompact}
L(t,\tau)v_\tau=\int_{\tau}^{t}U(t,s)F(s,S(s,\tau)v_\tau)\textrm{d}s.
\end{equation}

We now show that $\{S(t,\tau): t\geq\tau\}$ is pullback asymptotically compact in $E^{\frac{1}{2}}$ using \cite[Theorem 2.37]{NolascoAtratores}. For this we show that $U(s,\tau)$ decays as $t-\tau \to \infty$ and that $L(t,\tau)$ is compact for all $t>\tau$.

\begin{lemma}
\label{prop_pullback assintoticamente_compacto_atrator}
Suppose that the hypotheses \textbf{(H1)}, \textbf{(H2)}, \textbf{(H3)} and \textbf{(H4)} hold. Then, there exist constant $b>0$ and $M>0$ such that
$$
\left\|U(t,\tau)\right\|_{\mathcal{L}(E^{\frac{1}{2}})} \leq M e^{-b(t-\tau)},\quad \mbox{for all } t\geq \tau.
$$
Moreover, $L(t,\tau)$ is a  compact operator from $E^{\frac{1}{2}}$ into itself for all $t>\tau$. In particular, $\{S(t,\tau): t\geq\tau\}$ is pullback asymptotically compact in  $E^{\frac{1}{2}}$.
\end{lemma}

\begin{proof}
From Lemma \ref{decaimentoexponencial}, the evolution process $U(t,\tau)$ associated to the homogeneous problem
(\ref{equacaonoabstratanofixohomogeneo_S}) has exponential decay, that is, there exist constant $b>0$ and $M>0$ such that
$$
\left\|U(t,\tau)w\right\|_{E^{\frac{1}{2}}} \leq Me^{-b(t-\tau)}\left\|w\right\|_{E^{\frac{1}{2}}}, \quad t\geq \tau,
$$
for all $w\in E^{\frac{1}{2}}$.

Now, let us analyze the operator $L(t,\tau)$ given by (\ref{Lcompact}). Using Lemmas \ref{limitacaoaplicacaoabstrata} and   \ref{decaimentoexponencial}, we have 
\begin{eqnarray}
\nonumber \displaystyle
\left\|L(t,\tau)v_\tau\right\|_{E^{\alpha}}& \leq & \int_{\tau}^{t}  \left\|U(t,s)F(s,S(s,\tau)v_\tau)\right\|_{E^{\alpha}}\textrm{d}s \\
%\nonumber \displaystyle & \leq & M_{\alpha}\int_{\tau}^{t}(t-s)^{-\alpha}e^{-b(t-s)}\left\|F(s,S(s,\tau)v_\tau)\right\|_{E^0}  \textrm{d}s \\
\nonumber \displaystyle & \leq & M_{\alpha}\int_{\tau}^{t}(t-s)^{-\alpha}e^{-b(t-s)}(K_{1}\left\|v(s,\tau;v_{\tau})\right\|_{E^\frac{1}{2}}+K_{2})  \textrm{d}s,
\end{eqnarray} 
for all $\displaystyle \frac{1}{2}\leq\alpha<1$.

By Lemma \ref{limitacaotinfinitodassoluções}, if
$\left\|v_{\tau}\right\|_{E^{\frac{1}{2}}}\leq r$ for some $r>0$, then  $v(s,\tau;v_{\tau})$ is bounded in $E^{\frac{1}{2}}$, in the past, uniformly for all
$s \in \mathbb{R}.$ Thus, there exists a constant $R>0$ independent of $s \in \mathbb{R}$ such that
$$
\left\|L(t,\tau)v_\tau\right\|_{E^{\alpha}}\leq \displaystyle M_{\alpha}(K_{1}R+K_{2})b^{\alpha-1} \Gamma(1-\alpha).
$$
So, $L(t,\tau)$ takes bounded sets of $E^{\frac{1}{2}}$ into bounded sets of $E^{\alpha}$.

Since, the embedding $E^{\alpha} \hookrightarrow E^{\frac{1}{2}}$ is compact, for $\displaystyle \alpha>\frac{1}{2}$, then
we conclude that $L(t,\tau)$ is a compact operator
in $E^{\frac{1}{2}}$ for all $t>\tau$.

Therefore, by \cite[Theorem 2.37]{NolascoAtratores}, $\{S(t,\tau): t\geq\tau\}$ is strongly pullback asymptotically compact in $E^{\frac{1}{2}}$ and, consequently, it is also
pullback asymptotically compact in $E^{\frac{1}{2}}$.

\end{proof}

Finally, we conclude the main result of this work.

\begin{theorem}
\label{teo_existencia_pullback_atrator}
Suppose that the hypotheses \textbf{(H1)}, \textbf{(H2)}, \textbf{(H3)} and \textbf{(H4)} hold. Then, $\{S(t,\tau): t\geq\tau\}$ has a pullback attractor  $\{\mathcal{A}(t): t\in\mathbb{R}\}$  in $E^{\frac{1}{2}}$ and
\begin{equation}
\label{NOVOregattractor}
\bigcup_{t\in \mathbb{R}}\mathcal{A}(t)\; \text{is bounded in $E^{\frac{1}{2}}$}.
\end{equation}
\end{theorem}
\begin{proof}
By Lemmas \ref{limitacaotinfinitodassoluções} and \ref{prop_pullback assintoticamente_compacto_atrator}, $\{S(t,\tau): t\geq\tau\}$ is strongly pullback bounded dissipative and pullback asymptotically compact. Thus applying \cite[Teorema 2.23]{NolascoAtratores}, we have $\{S(t,\tau): t\geq\tau\}$  has a pullback attractor in $E^{\frac{1}{2}}$. Now, observe that due to (\ref{desigualdade_S_fortemente_limitado_dissipativo}) we concluded  (\ref{NOVOregattractor}).

\end{proof}

%%%%%%%%%%%%%%%%%%%%%%%%%%%%%

%%%%%%%%%%%%%%%%%%%%%%%%%%%%%

\section{Example} \label{Sec_Final_Ex_Difeo}

\vspace{0.2cm}

Now, we will see an application in $\mathbb{R}^3$ 
with linear and separable diffeomorphism, that is, for a spatially linear domain transformation, that satisfies the conditions of this work. 

For that, consider the domain in $\mathbb{R}^3$ given by
$$\mathcal{O}=\{(y_1,y_2,y_{3})\in \mathbb{R}^3:\; y^2_1 +y^2_2+y_{3}^{2} < 1 \}$$
and the diffeomorphism given by
\begin{equation}
    \nonumber
    \begin{array}{rccl}
    \displaystyle r :&   \mathbb{R} \times \overline{\mathcal{O}} &\rightarrow& \mathbb{R}^{3} \\ 
    \displaystyle 	 &   (t,y_1,y_2,y_{3}) & \mapsto & \displaystyle r(t,y_1,y_2,y_{3})=\frac{1}{e^{-t^{2}}+1}(y_1,y_2,y_{3})
    \end{array}
    \end{equation}
with inverse
$$r^{-1}(t,x_1,x_2,x_{3})=(e^{-t^{2}}+1)(x_1,x_2,x_{3}).$$

It is easy to verify that the diffeomorphism above satisfies the hypotheses \textbf{(H1)-(H4)} of this paper.

Note that some works only consider this particular case of diffeomorphism. For example, in the case of a  
 semilinear heat equation with
homogeneous Dirichlet boundary conditions, this type of diffeomorphism appears in an application in \cite{PeterKloeden}. For the case of a wave equation  with
homogeneous Dirichlet boundary conditions, all the work
\cite{Matofu} is carried out for this type of diffeomorphism and with domain in
$\mathbb{R}^3$.

%%%%%%%%%%%%%%%%%%%%%%%%%%%%%%
%\bibliographystyle{unsrt}% citação bibliográfica alpha
%\bibliography{bibliografia}  % associado ao arquivo:https://pt.overleaf.com/project/5f178a8ca4ad1b000118ec35 'bibliografia.bib

\end{document}